\documentclass[11pt, reqno]{amsart}
\usepackage{amsmath,amssymb,amsfonts,amstext, amsthm, amscd}
\usepackage{hyperref}
\usepackage{verbatim}
\usepackage{graphicx}
\usepackage{color}
\usepackage{enumerate}
\usepackage{epstopdf}

\newtheorem{theorem}{Theorem}[section]
\newtheorem{lemma}[theorem]{Lemma}
\newtheorem{proposition}[theorem]{Proposition}
\newtheorem{corollary}[theorem]{Corollary}

\newtheorem*{theorem-nonum}{Theorem 3}

\theoremstyle{definition}
\newtheorem*{definition*}{Definition}
\newtheorem{remark}[theorem]{Remark}
\newtheorem{example}[theorem]{Example}

\newcommand\ns[1]{ \left\{ {#1} \right\} }

\newcommand{\Z}{{\mathbb Z}}
\newcommand{\R}{{\mathbb R}}
\newcommand{\N}{{\mathbb N}}

\newcommand{\e}{\varepsilon}

\newcommand{\F}{\mathcal{F}}

\newcommand{\borel}{\mathcal{B}}

\newcommand\m{{\mu}}

\newcommand\B{{\mathcal B}}        

\newcommand\fracc[2]{ #1 /#2 }      
\newcommand\X{\Omega}

\newcommand\garbage[1]{}

\newcommand\kmark{{k_{\mathrm{mark}}}}

\renewcommand{\P}{{\mathbb P}}

\newcommand{\dff}[1]{\textbf{\emph{#1}}}

\newcommand{\ronn}{|}

\newcommand{\erk}{\hfill \ensuremath{\Diamond}} 

\DeclareMathOperator{\orb}{orb}

\usepackage{bbm}

\newcommand{\PrA}{\mathrm{Prob}(A)}

\newcommand{\inprod}{\otimes}
\newcommand{\inseq}{\oplus}
\newcommand{\emp}{\mathrm{emp}}
\newcommand{\Ty}{\mathrm{type}}

\newcommand{\dTV}{\mathrm{d}_\mathrm{TV}}
\newcommand{\DKL}{\mathrm{D}_\mathrm{KL}}

\newcommand{\Prob}{\mathrm{Prob}}

\newcommand{\rmu}{\rho}

\newcommand{\vol}{{\mathrm{vol}}}

\newcommand{\Diff}{{\mathrm{Diff}}}

\newcommand{\Ad}{\mathcal{A}}

\newcommand{\Zg}{\Gamma}

\newcommand{\Leb}{\mathrm{Leb}}

\newcommand{\id}{\mathrm{id}}

\newcommand{\pij}{\mathrm{proj}}

\newcommand{\muphys}{\mu_{\mathrm{ph}}}

\newcommand{\hata}{\hat{g}}

\begin{document}
	\author{Zemer Kosloff and Terry Soo}
	\title[Factors of nonsingular systems]{Sinai factors of nonsingular systems: Bernoulli shifts and Anosov flows}

	\address{Einstein Institute of Mathematics, 
		Hebrew University of Jerusalem, 
		 Edmund J. Safra Campus, Givat Ram. Jerusalem, 9190401, Israel.}
	\email{zemer.kosloff@mail.huji.ac.il}
	\urladdr{http://math.huji.ac.il/~zemkos/}
	\thanks{ZK is partially supported by Israel Science Foundation grant 1180/22.}
	\thanks{The version of record is published in the \emph{Journal of Modern Dynamics}, 20:597-634, 2024; DOI:10.3934/jmd.2024016}
	
	\address{Department of Statistical Science \\
		University College London. 
		}
	\email{t.soo@ucl.ac.uk; math@terrysoo.com}
	\urladdr{http://www.terrysoo.com}
	
	
\maketitle

\begin{abstract}
We show that a totally dissipative system has all nonsingular systems as factors, but  that this is no longer true when the factor maps are required to be finitary.  In particular, if a   nonsingular  Bernoulli shift  has a limiting marginal distribution $p$, then it cannot have,  as a finitary factor, an  independent and identically distributed (iid) system of entropy larger than $H(p)$; on the other hand, we show that  
iid systems with entropy strictly lower than $H(p)$ can be obtained as finitary factors of these Bernoulli shifts, extending Keane and Smorodinsky's  finitary version of Sinai's factor theorem to the nonsingular setting.  As a consequence of our results we also obtain that every transitive twice continuously differentiable Anosov diffeomorphism on a compact manifold, endowed with volume measure, has iid factors, and if the factor is required to be finitary, then  the iid factor cannot have Kolmogorov-Sinai entropy greater than the  measure-theoretic entropy of a Sinai-Ruelle-Bowen measure associated  with the Anosov diffeomorphism.
	\end{abstract}

\section{Introduction}

Let $(\Omega, \F, \mu)$ be a $\sigma$-finite measure space.  We will often be concerned with the case where $\mu$ is a probability measure, so that $\mu(\Omega)=1$.     Let $T : \Omega \to \Omega$.  The map $T$ is \dff{ergodic} if $\mu( E \triangle T^{-1}(E) )=0$ implies that $0 \in \ns{\mu(E), \mu(\Omega / E)}$.   We say that $T$ is \dff{nonsingular} if $\mu \circ T^{-1} \sim \mu$, that is, the measures have the same null sets; in the case that $\mu\circ T^{-1} =\mu$ we say that $T$ is \dff{measure-preserving}.    We refer to $(\Omega, \F, \mu, T)$ as a  \dff{nonsingular dynamical system}.   In the case where $T$ is ergodic and  probability-preserving,  Kolmogorov-Sinai entropy \cite{MR3930576,MR0103255, MR0103256}, a single nonnegative real number, is assigned  to the dynamical system and measures the amount of randomness contained in the system, all of which can be accounted for by a Bernoulli subsystem in the following way.

Let $A$ be a finite set of symbols,  $(\rho_i)_{i \in \Z}$ be a sequence of probability measures on $A$, and $\nu= \bigotimes_{i\in \Z} \rho_i$ be the product measure on the sequence space
 $A^{\mathbb{Z}}$ 
 endowed with the usual product topology and the Borel product sigma-algebra   $\borel= \borel(A^\Z)$.    When the product probability space $(A^{\Z}, \borel, \nu)$  is endowed with the  \dff{left-shift}  $S:A^{\Z} \to A^{\Z}$  given by $(Sx)_i = x_{i+1}$, we refer to the dynamical system as a   \dff{Bernoulli shift}; if all the measures $\rho_i \equiv  {p}$ are identical, then we say that system is  \dff{independent and identically distributed (iid)}; sometimes we will also refer to iid systems as \dff{stationary} Bernoulli shifts.    Thus in the nonsingular setting, a Bernoulli shift is a sequence of independent, but not necessarily identical, random variables.  Sometimes we will simply refer to  a system by its corresponding measure or endowed mapping.  
The Kolmogorov-Sinai entropy of iid systems is given by the usual Shannon entropy \cite{MR26286}:   $H({p}):=-\sum_{a \in A} {p}(a) \log {p}(a)$.    We say that a nonsingular system
$(\Omega, \mathcal{F}, \mu, T)$ 
 has an \dff{iid factor} of entropy $h'$ if
  there exists a \dff{factor map} $\phi: \Omega \to A^{\Z}$  such that the mapping is equivariant,  $\phi \circ T = S \circ \phi$, and $\mu \circ \phi^{-1} \sim q^{\Z}$, where $q$ is a probability measure on $A$ with $H(q) = h'$; we emphasize that in the nonsingular case, we allow the possibility that  the push-forward of $\mu$ under $\phi$ is \emph{not} exactly $q^{\Z}$, only that it is a measure that is equivalent to $q^{\Z}$.    The  Sinai factor theorem \cite{Sinai} gives that an ergodic probability-preserving system with entropy $h$ has all  iid factors of entropies less than or equal to $h$.

A common notion of \emph{chaos} for nonsingular dynamical systems is that the system can be used to simulate a bi-infinite collection of iid coin tosses and the Sinai factor theorem shows that the chaotic probability-preserving systems are precisely the systems with positive Kolmogorov-Sinai entropy \cite{MR1023980}. 
In the absence of a similarly robust  notion of entropy and entropy theory, it is unclear which nonsingular systems are chaotic and how to define 
entropy.  
We show that all totally dissipative systems are chaotic and that the analogue of Keane and Smorodinsky's  finitary version of Sinai's factor theorem \cite{keanea} remains true for a large class of nonsingular Bernoulli shifts.

\subsection{Dissipative systems are universal} 
Recently, we showed that a large class  of nonsingular Bernoulli shifts have stationary Bernoulli shift factors \cite{KosSoo20, KoSooFactors}.   A curious aspect of our results is that it also applies to dissipative Bernoulli shifts, and it turns out there was a deeper reason why we never had to assume conservativity.   We say that a nonsingular system is \dff{factor-universal} if it has every nonsingular system as a factor.   Recall that a nonsingular system is dissipative if it has a wandering set of positive measure, is conservative if there are no wandering sets of positive measure,  and is totally dissipative if the measurable union of the wandering sets is the whole space; see  Section \ref{diss-section} for more precise definitions.  
\begin{theorem}
\label{thm: dissipative factors}
A totally dissipative system on a non-atomic measure space is factor-universal.
\end{theorem}
We prove  Theorem \ref{thm: dissipative factors} in the slightly more general setting of a countable group action; see Theorem \ref{Dissipative actions}.    We will also consider a version of
Theorem \ref{thm: dissipative factors} 
 in the setting of a flow in  Theorem \ref{Dissipative flows}; the case of flows is more subtle, and it is not true that every totally dissipative flow is factor-universal.  

Theorem \ref{thm: dissipative factors}  may come as a surprise, 
since a totally dissipative transformation is, by a result of Hopf, isomorphic in the nonsingular category to the simple shift  transformation $x \mapsto x+1$ equipped with Lebesgue measure on the real line, and the latter is a rather predictable  and seemingly uninteresting system; see Proposition \ref{Hopf}.   Although the proof is surprisingly simple, our result has several interesting applications in bridging several definitions of chaos in dynamical systems.

In the context of 
Ilya  Prigogine's theory of dissipative structures \cite{doi:10.1126/science.201.4358.777}, Theorem \ref{thm: dissipative factors} provides an abstract mathematical explanation of how dissipative structures are systems which can possess both chaotic behavior and elliptic islands, since under this mathematical framework  \emph{any} system is a subsystem of a dissipative system.

Theorem \ref{thm: dissipative factors} together with the works of Sinai and Livsic \cite{MR0317355} and Gurevic and Oseledec \cite{MR0320274},  implies that \emph{every} transitive $C^2$ Anosov diffeomorphism has stationary Bernoulli factors, and thus all transitive $C^2$ Anosov diffeomorphisms can be used to simulate an iid sequence, in an equivariant way. 

\begin{theorem}
\label{Anosov diffeo chaos}
	A transitive $C^2$ Anosov diffeomorphism on a compact manifold, endowed with the natural volume measure, has iid factors, which may be chosen to have infinite entropy in the case the diffeomorphism is dissipative.   
\end{theorem}

See Section \ref{AN-section} for definitions and more details; see also Section \ref{ANflow-section} and  Theorem \ref{Anosov flow chaos} for the case of an Anosov flow.     
We note that in  Theorem \ref{thm: dissipative factors} there is a key assumption that the system is totally dissipative, and that there is no regularity assumption on the associated factor mapping, other than measurabililty.    If the factor map is required to be continuous almost-surely, then we have the following upper bound on the entropy of an iid factor.    

\begin{theorem}[Upper bound for Anosov diffeomorphisms]
\label{anosov.bound}
	An iid system that is obtained as a   continuous almost-surely factor of a transitive $C^2$ Anosov diffeomorphism on a compact manifold, endowed with the natural volume measure,  will have Kolmogorov-Sinai entropy that is bounded by the measure-theoretic entropy of a Sinai-Ruelle-Bowen measure associated with the Anosov diffeomorphism.
\end{theorem}

See Section \ref{section:anosov-bound} for details.  In this context, a factor map  that is continuous almost-surely is sometimes  referred to as being  \dff{finitary}.    In some special cases, such as the measure-preserving case of a hyperbolic toral automorphism \cite{MR525312, keanec}, we know that the upper bound in  Theorem \ref{anosov.bound} is achieved; these results rely on symbolic dynamics and  the finitary constructions of  Keane and Smorodinsky \cite{keanea, keaneb} for stationary Bernoulli shifts, which we will extend to the nonsingular setting.

\subsection{Finitary factors of nonsingular Bernoulli shifts}

      When we turn our attention to the case of nonsingular Bernoulli shifts, and consider factor maps that are finitary \cite{MR2306207}, so that they are continuous almost surely, we have results that are consistent with those of entropy theory and Keane and Smorodinsky \cite{keanea}.  We note that there are both dissipative and conservative nonsingular Bernoulli shifts that do not have an absolutely continuous invariant probability measure, with the first examples of the latter given by Krengel \cite{MR0269808} and Hamachi \cite{MR662470}.  It follows from Theorem \ref{thm: dissipative factors} that many Bernoulli shifts have infinite entropy iid factors.  Specifically,   all dissipative nonsingular Bernoulli shifts are totally dissipative and thus shift universal; see also Example \ref{example-with-three} and Proposition \ref{example-no}.
         In what follows, we address together, dissipative and conservative nonsingular Bernoulli shifts, with the added finitary regularity assumption on the factor mapping.

 A product measure $\rho = \bigotimes_{i \in \Z} \rho_i$ on $A^\mathbb{Z}$ satisfies the \dff{Doeblin condition} if there exists $\delta>0$ such that for all $a\in A$ and $i\in\mathbb{Z}$, we have $\rho_i(a)>\delta$.       We say that the \dff{limiting marginal measure} is ${p}$ if $\rho_{|i|}(a) \to {p}(a)$ for all $a \in A$.   Motivated by Sinai's original factor theorem,   we say that a family of  iid factor maps   of this system have \dff{near optimal} entropy if they can produce systems of entropy $H({p}) - \e$ for every $\e >0$.  An iid factor has \dff{optimal} entropy if it has entropy $H(p)$ and has \dff{super optimal} entropy if it has entropy greater than $H(p)$. 

 Let $A$ and $B$ be finite sets.  Consider a factor map $\phi:A^{\Z} \to B^{\Z}$, where $A^{\Z}$ is equipped with a nonsingular measure $\mu$.  Consider the zeroth coordinate projection $\bar{\phi}$ given by $\bar{\phi}(x) = \phi(x)(0)$.    We say that $\phi$ is \dff{finitary} if for every $b \in B$ there exists $C_b$ which is a countable union of cylinder subsets of $A^{\Z}$ such that $\mu( \bar{\phi}^{-1}(b) \triangle C_b) =0$.   This condition is equivalent to the map being continuous almost surely with respect to $\mu$ and also equivalent to $\phi$ having an almost surely finite coding radius, see for example \cite[page 281]{MR1073173}.  This condition was first introduced by Weiss \cite{MR0419738}  and studied by Denker and Keane \cite{MR571401}, see also \cite{MR2306207}.

\begin{theorem}[Upper bound for finitary factors]
	\label{bound-Kalikow}
	A nonsingular Bernoulli shift that has a limiting measure,   does not have a finitary iid factor with super optimal entropy.   
\end{theorem}

We will prove Theorem \ref{bound-Kalikow} as a consequence of a  more general result, Theorem \ref{Kalikow general}, which by using symbolic dynamics \cite{MR0233038} and some finer  results regarding  Anosov diffeomorphisms \cite{MR380889},  will also allow us to infer Theorem \ref{anosov.bound}.

Turning our attention to positive results,  we recently proved the following   theorem regarding finitary factors in \cite[Theorem 1]{KoSooFactors}.

\begin{theorem}[Low entropy factors \cite{KoSooFactors}]
\label{studia}
 Every nonsingular Bernoulli shift  which  satisfies the Doeblin condition has a finitary iid factor. 
 \end{theorem}

In Theorem \ref{studia}, it is obvious from our construction that the finitary factor is non-optimal with respect to entropy. Keane and Smorodinsky \cite{keanea, keaneb} defined finitary factor isomorphisms between any two measure-preserving Bernoulli shifts of finite entropy using a \emph{marker-filler} method; our proof of Theorem  \ref{near} will also make use of these ideas adapted to the nonsingular setting.    

\begin{theorem}[Near optimal entropy factors]
\label{near}
If a nonsingular Bernoulli shift  satisfies the Doeblin condition, and has a limiting measure, then it  has near optimal entropy finitary iid factors. 
\end{theorem}

By an elementary argument, we obtain the following extension, where we remove the Doeblin condition and allow for the possibility of a countable number of symbols.

\begin{corollary}
\label{countable}
If a nonsingular Bernoulli shift on a  countable 
(or finite) 
number of symbols  has a limiting measure, then it  has near optimal entropy finitary iid factors. 
\end{corollary}

In Theorem \ref{countable}, in the case that the Bernoulli shift has infinite entropy, the finitary iid factors can be taken to have arbitrarily large finite entropy; see Section \ref{section-countable} for details.

We  remark that there are many conservative Bernoulli shifts which do not have an absolutely continuous invariant measure, and understanding these shifts is an active area of research; see for example, \cite{BjoKosVaes, MR662470,KosKMaharam,DaniLem}.   Specifically, Theorem \ref{near} does not follow from  Theorem \ref{thm: dissipative factors} or  Keane and Smorodinsky's version of the Sinai factor theorem.

\begin{example}
\label{example-with-three}
We illustrate our theory with following example which was studied in \cite{VaesWahl}, where the theory provides simple examples of nonsingular conservative Bernoulli shifts that do not admit an absolutely continuous invariant measure.
Consider  $A = \ns{0,1}$, and the family of product measures $\rho^c$ with marginals
\begin{equation}
	\label{sqrt-VW}
	\rho_n^c(0) = \frac{1}{2} + \frac{c}{\sqrt{n}} \cdot \mathbf{1}[n \geq 1, c/\sqrt{n} < 1/2],
\end{equation}
where $c>0$ is a parameter. 
It is easy to  check using Kakutani's theorem (see Section \ref{earlier-dis}) that   $\rho^c$ is nonsingular and is not 
equivalent to  the product measure $(\tfrac{1}{2}, \tfrac{1}{2})^{\Z}$, but by Theorem \ref{near}, it still has a finitary  iid factor of entropy almost $H(\tfrac{1}{2},\tfrac{1}{2})=\log 2$ and  by Theorem \ref{bound-Kalikow} finitary iid factors must have entropy at most $\log 2$.  

We proved that there exists a critical  $c^{*} \in(1/6, \infty)$  such that  the shift is totally dissipative for $c> c^{*}$  and for $c  < c^{*}$ the shift is conservative \cite[Theorem 3]{KoSooFactors}.  Thus  when $c >c^{*}$, by Theorem \ref{thm: dissipative factors},  the shift has any system 
(even a circle rotation) 
as a (non-finitary) factor.  \erk
\end{example}

With regards to nonsingular Bernoulli shifts, we do not know if  finitary optimal entropy factors must exist, and 	   we cannot exclude the possibility of super optimal entropy non-finitary factors, even when the shift is conservative.

A  nonsingular dissipative Bernoulli shift that does \emph{not} satisfy the Doeblin condition and has a trivial limiting measure may have the following varied behaviour.   
\begin{proposition}
\label{example-no}
There exists a nonsingular Bernoulli shift which has all iid factors, but no finitary iid factor. 
\end{proposition}

Our proof of Proposition \ref{example-no} will follow from Theorem \ref{thm: dissipative factors} and Theorem \ref{bound-Kalikow}.   See Section \ref{example-section}.

\subsection{Some remarks about random number generation and Theorem \ref{near}}

The practical problem of  random number generation has a long history \cite{8247790, vonNeumann1951} and continues to be a topic of current research.  A central goal of (true) random number generation is to generate iid fair coin tosses from an iid biased source, in an efficient way, see for example \cite{MR1150365}.   A somewhat more realistic, and less studied assumption is to assume the source is \emph{noisy}, that is  given by  independent, but not identical random variables \cite{6033845}; the goal in this setting is often to efficiently  generate  independent bits that are \emph{approximately} fair.   In the  language of random number generation,  we are interested in generating,  in an equivariant way,   from a noisy source,  a sequence that is statistically  indistinguishable from an iid sequence of near optimal entropy.  The equivariance  requirement means that if the source input is given on a ticker tape, then the same procedure is applied everywhere on the ticker tape, and positions on the ticker tape do not need to be additionally labelled. The factor map given by Theorem \ref{near} is finitary, so that our procedure, in principle, can be implemented by a  machine.

\subsection{Some tools used in our proofs} 
Our proof of Theorem \ref{thm: dissipative factors} for countable group actions is simple and given in the next section.   The corresponding version for flows is more involved, and as a result the version of Theorem \ref{Anosov diffeo chaos} for Anosov flows is a bit harder and requires  understanding of their ergodic decompositions.


In Section \ref{section-near}, our proof of Theorem \ref{near}  draws from a combination of ideas of  Keane and Smorodinsky \cite{keanea, keaneb}, Harvey, Holroyd, Peres, and Romik \cite{MR2308235},   Karen Ball \cite{ball},  some tools from information theory,  and our last paper on factors in the nonsingular setting \cite{KoSooFactors}.   
 In \cite{MR2308235}, the authors used marker-filler methods in combination of with a version of Elias' \cite{elias} construction of an unbiased (random length) iid binary sequence from a stationary source   to produce \emph{source universal} finitary factors.     Ball also employed marker-filler methods in her construction of  \emph{monotone} finitary factors; see also \cite{Qsc}.    We will prove a  disintegration of a non-stationary product measure that will have enough regularity which will allow us to recover enough stationarity.

In Section \ref{kalikow-section},  our proof of Theorem \ref{bound-Kalikow} will make use of some ideas  from Kalikow's \cite[Chapter 4.1, Theorem 365]{Kalikow} beautiful elementary proof of Kolmogorov's theorem \cite{MR0103254,MR0103255} that the three-shift  is not a factor of  the two-shift; whereas Kolmogorov's proof was an easy consequence that entropy cannot increase under factor maps, Kalikow's proof has the advantage that it  uses ``essentially nothing.''  
This proof can be modified and extended to our nonsingular setting, at the cost of restricting to finitary factors, leaving the possibility that the analogue of Kolmogorov's theorem in the nonsingular setting does not hold, even when the shift is conservative.      We will prove an abstraction of Theorem \ref{bound-Kalikow} which will also apply to other symbolic systems.  In order to obtain the bound for Anosov diffeomorphisms, we use  symbolic dynamics, and carry out entropy calculations that  simultaneously involve two measures; the first volume lemma of Bowen and Ruelle \cite{MR380889}, allows us to  reconcile the volume and its  Sinai-Ruelle-Bowen measure.

\section{Dissipative actions are factor universal}
\label{diss-section}

Let $(\X,\F,\mu)$ be a non-atomic (standard) $\sigma$-finite measure space. An (nonsingular) \dff{automorphism} of $(\X,\F,\mu)$ is an invertible measurable map $R:\X\to \X$ which satisfies $\mu\circ R^{-1}\sim \mu$. Denote the automorphism group of $(\X,\F,\mu)$ by $\mathrm{Aut}(\X,\F,\mu)$. A  \dff{$G$-action} of a locally compact group $G$ on $(\X,\F,\mu)$, sometimes abbreviated as $G\curvearrowright (\X,\F,\mu)$, is a group homeomorphism $g\mapsto T_g$ from $G$ to $\mathrm{Aut}(\X,\F,\mu)$.

 The action $(\X',\F',\mu', (S_g)_{g \in G})$ is a \dff{factor} of $(\X,\F,\mu, (T_g)_{g \in G})$  if there exists a measurable map $\pi:\X\to \X'$ such that $\pi$ is equivariant, so that $\mu$ almost everywhere, for all $g\in G$, we have $\pi\circ T_g=S_g\circ \pi$, 
 and in addition 
 $\mu'\sim \mu\circ \pi^{-1}$; 
 if $\pi$ is invertible, modulo null sets, then we say that the two systems are \dff{isomorphic}.  We say that the action $G\curvearrowright (\X,\F,\mu)$ is \dff{factor-universal} if every $G$-action $G\curvearrowright (\X',\F',\mu')$ is a factor.

\subsection{Discrete group actions} 
 Let $G$ be a countable group. Given an action $G\curvearrowright (\X,\F,\mu)$, a set $W\in\F$ is \dff{wandering} if $\left\{T_gW\right\}_{g\in G}$ are pairwise disjoint;  the action is \dff{dissipative} if there exists a wandering set with $\mu(W)>0$ and it is \dff{totally dissipative} if there is a wandering set $W$ such that the  union of its translates is the whole space modulo a null set: $\biguplus_{g\in G}T_gW=\X \bmod \mu$.  The action is \dff{conservative} if every wandering set is of measure zero. A set $A\in \F$  is \dff{$G$-invariant} if for all $g\in G$, we have $T_g^{-1}A=A$.   
 

The following action is the prototypical example of a totally dissipative action of a non-atomic measure space. Let $Z=[0,1]\times G$ and $\nu$ be the product measure of the uniform measure on $[0,1]$ and the counting measure on $G$. The \dff{translation action} of $G$ on $Z$ is defined by $S_g(x,h)=(x,hg)$. The following result is due to Hopf, and closely related to the Hopf decomposition which partitions the space into a dissipative part and a conservative part; see \cite[Propostion 1.1.2 and Exercise 1.2.1]{AaroBook}.
\begin{proposition}[Hopf]
	\label{Hopf}
	If  $G\curvearrowright (\X,\F,\mu)$ is  a totally dissipative action of a non-atomic measure space, then is isomorphic to the translation action of $G$ on $[0,1]\times G$. 
\end{proposition} 
\begin{proof}
By passing to an absolutely continuous probability we may assume that $\mu(\Omega)=1$. Since the action is totally dissipative, let $W\in \F$ be a wandering set whose disjoint union of translates are the whole space.     Recall we assume that the measure space is standard and  non-atomic and $\m(W)>0$.   By the Borel isomorphism theorem \cite[Theorem 3.4.23]{borel}, let $\beta:W\to [0,1]$ be a bijective measurable map such that $\tfrac{d\mu\circ \beta}{d\mu}=\tfrac{1}{\mu(W)}$.   Define $\pi:\X\to [0,1]\times G$ by $\pi(x)=\left(\beta\left(T_g^{-1}x\right), g\right)$ where $g\in G$ is the unique group element such that $x\in T_gW$.  Clearly, $\pi$ is bijective and equivariant from which it follows that $\pi$ is an isomorphism.  
\end{proof}
\begin{theorem}
\label{Dissipative actions}
A totally dissipative action of a countable group on a non-atomic measure space is factor universal. 
\end{theorem}
\begin{proof}
Let the first coordinate,  $\left(\X_1,\F_1 ,\mu_1,(T_g)_{g\in G}\right)$ be a totally dissipative $G$-action of a non-atomic measure space, and let the second coordinate,  $\left(\X_2,\F_2,\mu_2, (S_g)_{g\in G}\right)$ be any nonsingular $G$-action. Consider the direct product action $(T_g\otimes S_g)_{g\in G}$, which is a nonsingular $G$-action on $\left(\X_1\times \X_2,\F_1\otimes \F_2,\mu_1\otimes \mu_2\right)$. Clearly the projection $\pij_{2}$  from $\X_1\times \X_2$ to the second coordinate is a factor map.


It is easy to verify that the product action inherits the   totally dissipative property from the first coordinate.  By Proposition \ref{Hopf}, there exists an isomorphism $\psi:\X_1 \to \X_1\times \X_2$ between the first  action and  the direct product action. We deduce that the composition $\pij_{2} \circ \psi$ illustrated by
\begin{equation*}
\X_1\xrightarrow{\psi} (\X_1\times \X_2) \xrightarrow{\pij_{2}} \X_2,
\end{equation*}
 is the desired factor map from the first coordinate to the second.
\end{proof}
The following shows that the latter are the only factor universal actions. 
\begin{proposition}
If $G\curvearrowright (\X,\F,\mu)$ has a totally dissipative factor, then it is totally dissipative. 
\end{proposition}
\begin{proof}
Towards a contradiction, suppose that the action is not totally dissipative and it has a totally dissipative factor $(\X',\F',\mu',(S_g)_{g\in G})$, via the factor map $\pi$.   Since the conservative part in the Hopf decomposition is $G$-invariant, it follows that by restricting the factor map to
 the conservative part
 we may assume without loss of generality that  $G\curvearrowright (\X,\F,\mu)$ is a conservative action.  

Choose some $W\in \F'$ a wandering set for $(S_g)_{g\in G}$ of positive $\mu'$-measure. It is easy to see that set $\pi^{-1}W\in \F$ is a wandering set of positive measure, contradicting the conservativity of  $G\curvearrowright (\X,\F,\mu)$.
\end{proof}
\subsection{$\mathbb{R}$ and $\mathbb{R}^d$-flows}
A nonsingular $\mathbb{R}^d$-action $\left(\X,\F,\mu,\left(\phi_t\right)_{t\in\R^d}\right)$ is a \textbf{$\R^d$-flow}, and the case $d=1$ will simply be referred to as a \dff{flow}.    An $\mathbb{R}^d$-flow is \dff{totally dissipative} if for every $s>0$, we have that $\left(\phi_{sn}\right)_{n\in\Z^d}$ is a totally dissipative $s\Z^d$ action; this is equivalent to verifying the dissipativity for the single case of $s=1$ \cite[Corollary 1.6.5]{AaroBook}.      Sometimes we will denote Lebesgue measure on $\mathbb{R}^d$ by $\Leb = \Leb_{\R^d}$.  Let $(\Gamma, \mathcal{C}, \nu)$ be a standard probability space \cite{MR0047744}.   A \dff{translation flow of $\R^d$ with respect to $\Gamma$}  is the dissipative action on $\Gamma \times \R^d$, equipped with  the product measure  $\nu \otimes \Leb$, given by  $\tau_s(z,t)=(z,t+s)$, for all $s,t \in \R^d$.       When $\Gamma$ is the unit interval and $\nu$ is Lebesgue measure, then we say that translation flow is \dff{canonical}.  Note that when $\Gamma = \ns{\ast}$ is a one point space, the translation flow is ergodic, and we may omit $\Gamma$;  by the theorem below, all ergodic dissipative flows are isomorphic to this translation flow.  
 
The following analogue of Proposition \ref{Hopf} was proved by Krengel \cite[Satz 4.2]{MR296254} for flows and by Rosinski \cite[Theorem 2.2]{MR1813849} for $\R^d$-flows by constructing a relevant cross section.
\begin{theorem}[Krengel and Rosinski]
\label{Krengel}
For every totally dissipative $\R^d$-flow on a nonatomic measure space, there exists a standard probability space, $(\Zg,\mathcal{C},\nu)$ such that the flow is isomorphic to the translation flow  of $\R^d$ with respect to $\Gamma$.
\end{theorem}

	We note that in  Rosinski \cite{MR1813849},  $(\Zg,\mathcal{C},\nu)$ is a $\sigma$-finite standard, measure space.  By passing to a $\nu$ equivalent probability measure, we may assume that $\nu$ is a probability measure since the identity map gives the obvious isomorphism.    In addition, we may substitute for $(\Zg,\mathcal{C},\nu)$, any other probability space that is isomorphic in the category of measure.  
	 The space $\Zg$ corresponds to the ergodic decomposition of the flow. In particular,  the translation flow $\tau$ is ergodic if and only if $\Zg$ is a one point space.
By Theorem \ref{Krengel}, an ergodic dissipative flow of a non-atomic measure space is isomorphic to the translation flow on $\R^d$ with respect to the one point space $\Gamma = \ns{\ast}$, which we refer to as \dff{the ergodic dissipative $\R^d$-flow}.
\begin{theorem}[Factors of totally dissipative flows]
\label{Dissipative flows}
\hspace{2 cm}
\begin{enumerate}[(a)]
	\item  
	The canonical   translation flow of $\R^d$ is factor-universal. 
	\item 
	The ergodic dissipative $\R^d$-flow is a factor of any totally dissipative $\R^d$-flow on a non-atomic measure space.  
\end{enumerate}
\end{theorem}

\begin{proof}
For the proof of part (a),  let $(\X,\F,\mu,(\phi_s)_{s\in \R^d})$ be a given  nonsingular flow, which may not be dissipative.     Let $\tau$ be the canonical translation flow.  The product flow $\phi\otimes \tau = (\phi_s \otimes \tau_s)_{s \in \R^d}$ is a totally dissipative flow on a non-atomic measure space. By Theorem \ref{Krengel}, there exists a standard probability space, $(\Zg,\mathcal{C},\nu)$ such that $\phi\otimes \tau$  is isomorphic to the translation flow of $\R^d$  with respect to $\Gamma$; let $\Theta: \X\times ([0,1]\times \R) \to \Gamma \times \R^d$ be the isomorphism.

By a routine variation of the Borel isomorphism theorem, there exists a  homomorphism of probability spaces $g:[0,1]\to \Gamma$ such that $\Leb_{[0,1]}\circ g^{-1}=\nu$.  Let $\pij_{\X}:\X\times ([0,1]\times \R^d)\to \X$ be the projection onto $\X$.  Then the following compositions of mappings given by the diagram
\[
[0,1]\times \R^d\xrightarrow[]{g\otimes \id_{\R^d}} \Zg\times \R \xrightarrow{\Theta^{-1}} \X\times ([0,1]\times \R^d) \xrightarrow{\pij_{\X}}\X, 
\]
gives a factor map from the canonical translation flow of $\R^d$ to the given flow. 

For part (b), again by Theorem  \ref{Krengel},  every dissipative $\R^d$-flow is isomorphic to the translation flow of $\R^d$  with respect to some standard probability space $\Gamma$.   Next, note that the projection $\pij_{\R^d}: \Gamma \times \R^d \to \R^d$ is  a factor map from the dissipative translation flow of $\R^d$ with respect to $\Gamma$ to the ergodic dissipative $\R^d$-flow. 
\end{proof}
\begin{remark}The case of flows differs  from the countable group case since the translation flow on $\R^d$ given by $\tau_s(x)=x+s$ is ergodic. Consequently, not all totally dissipative flows on a non-atomic measure space are factor universal.   We thank Jon Aaronson for pointing out a misinterpretation of Theorem \ref{Krengel} which led to an incorrect formulation of Theorem \ref{Dissipative flows} in an earlier version of this manuscript.  \erk
\end{remark}

\subsection{Applications to chaos in $C^2$ Anosov diffeomorphisms and flows}
\label{AN-section}

Let $M$ be a compact Riemannian  manifold without boundary and $\vol_M$ be the volume measure on $M$. By the change of variables formula, for every diffeomorphism $f$, the  system  $(M, \borel(M), \vol_M,f)$ is  nonsingular.   Let $\Diff^k(M)$ be the collection of $C^k$-diffeomorphisms. A diffeomorphism is  \dff{(topologically) transitive} if there exists $x\in M$ such that $\left\{f^n(x):\ n\in\Z\right\}$ is dense in $M$. Similarly a \dff{$C^2$-flow} is a homeomorphism $s\mapsto \phi_s$ from $\R$ to $\Diff^2(M)$ and the flow is transitive if it has a dense $\R$-orbit.   Anosov systems are a central object of study in ergodic theory and dynamical systems \cite{MR0224110, MR2423393,MR0228014}.

\subsubsection{Applications to Anosov diffeomorphisms}
A diffeomorphism $f:M\to M$ is  \dff{Anosov}  (uniformly hyperbolic)  if for all $x\in M$, there exists a decomposition of the tangent bundle over $x$, given by 
$T_xM=E_x^{\mathrm{s}}\oplus E_x^{\mathrm{u}}$ such that:
\begin{itemize}
	\item For all $x\in M$, we have $Df(x)E_x^{\mathrm{s}}=E_{f(x)}^{\mathrm{s}}$ and 
	$Df(x)E_x^{\mathrm{u}}=E_{f(x)}^{\mathrm{u}}$. 
	\item For any metric on $TM$, there exists $a>0$ and $\lambda\in (0,1)$ such that for all $v\in E_x^{\mathrm{s}}$, for all $n\in\Z^{+}$, we have
	\[
	\left\|Df^n(x)v\right\|\leq a\lambda^n\left\|v\right\|,
	\]
	and similarly for all 
	$v\in E_x^{\mathrm{u}}$, 
	for all $n\in\Z^{+}$, we have 
	$$\left\|Df^{-n}(x)v\right\|\leq a\lambda^n\left\|v\right\|.$$ 
\end{itemize}

Our proof of Theorem \ref{Anosov diffeo chaos}  is a consequence of  Livsic-Sinai  \cite{MR0317355},  when a  Sinai-Ruelle-Bowen (SRB) measure \cite{MR1933431} is available; on the other hand,   Gurevic and Oseledec \cite{MR0320274} showed that in the absence of a volume absolutely continuous invariant probability (a.c.i.p),  the Anosov diffeomorphism is dissipative and thus Theorem \ref{Dissipative actions} applies; see also Theorem \ref{GO}.    We remark that in the categorical sense  \emph{most} $C^2$ Anosov diffeomorphism do not have a volume a.c.i.p.\ \cite{MR0399421}; see also \cite[Corollary 4.15]{MR2423393}.

\begin{proof}[Proof of Theorem \ref{Anosov diffeo chaos}]
	Let $f$ be a $C^2$ Anosov diffeomorphism.  If there exists a volume absolutely continuous invariant probability (a.c.i.p)  $\mu$, then  $\mu$ is an SRB measure and $(M,\B(M),\mu,f)$ has positive  entropy \cite{MR0317355}.  By the Sinai factor theorem $(M,\B(M),\mu,f)$ has an iid factor, and since $\mu\sim \vol_M$, this remains true when $\mu$ is replaced by $\vol_M$.   
	
	In the absence of a volume a.c.i.p.,  $(M,\B(M),\vol_M,f)$ is a totally dissipative transformation of a non-atomic measure space  \cite{MR0320274}, and  it follows from Theorem \ref{Dissipative actions} that every stationary Bernoulli shift is a factor.  
\end{proof}

%
%
%
%
%
%

\subsubsection{Applications to Anosov flows}
\label{ANflow-section}
We will prove a version of Theorem \ref{Anosov diffeo chaos} for the case of flows.
Following Ornstein \cite{MR0318452}, we say that a probability-preserving flow $\left(\X,\F,\mu,\left(\phi_t\right)_{t\in\R^d}\right)$  is a \dff{Bernoulli flow} if for every $s>0$, the discrete-time system  $\left(\X,\F,\mu,\phi_{s}\right)$ is isomorphic to a stationary Bernoulli shift.   Bernoulli flows are the analogues of iid-systems for $\R$-flows and two Bernoulli flows whose time-one maps have equal Kolmogorov-Sinai entropy are isomorphic \cite{MR3052869,MR0318452}.  Canonical examples of Bernoulli flows include the Totoki flow \cite{MR0265554,totoki,MR0318452}, and infinite entropy flows arising from Brownian motions and  Poisson processes \cite{Orn-mix-type,MR910005}.

A $C^2$ flow $\left(\phi_t\right)_{t\in\R}$ on a  compact Riemannian  manifold  $M$ is an \dff{Anosov flow} if for all $x\in M$, the decomposition 
$T_xM=E_x^{\mathrm{s}}\oplus E_x^{\mathrm{c}}\oplus E_x^{\mathrm{u}}$   
is $D\phi_t$ equivariant, 
$E^{\mathrm{c}}$ 
is one dimensional and corresponds to the direction of the flow, and there exists $A>0$ and $b>0$ such that for all $x\in M$, 
\begin{align*}
\left\|D\phi_t v\right\|\leq Ae^{-bt}\left\|v\right\|, &\ \ \ \  \text{for every}\ t>0\ \text{and}\ v\in E_x^{\mathrm{s}}.\\
\left\|D\phi_{-t} v\right\|\leq Ae^{-bt}\left\|v\right\|, &\ \ \ \  \text{for every}\ t>0\ \text{and}\ v\in E_x^{\mathrm{u}}.
\end{align*}

\begin{theorem}
\label{Anosov flow chaos}
A transitive $C^2$  Anosov flow, endowed with the natural volume measure, has a Bernoulli flow as a factor.
\end{theorem}

Our proof of Theorem \ref{Anosov flow chaos} proceeds as in Theorem \ref{Anosov diffeo chaos}.  Sinai \cite{MR0399421} showed  that for every transitive  $C^2$ Anosov flow there exists two probability measures $\mu^+$ and $\mu^-$ such that for every continuous function $g:M\to \R$, for $\vol_M$-almost every $x \in M$, we have
\begin{align*}
\lim_{N\to\infty}\frac{1}{N}\int_0^N g\circ \phi_t(x)dt&=\int gd\mu^+,\ \ \text{and}\\
\lim_{N\to\infty}\frac{1}{N}\int_{-N}^0 g\circ \phi_t(x)dt&=\int gd\mu^-;
\end{align*}
see also \cite{MR380889} for the more general setting of \emph{Axiom A} flows. Again, the  measures $\mu^+$ and $\mu^-$ are called \dff{SRB measures}.    The following was proved by Gurevic and Oseledec for $C^2$ Anosov diffeomorphisms, and we present the identical proof for flows for completeness. 
\begin{theorem}[Gurevic and Oseledec]
\label{GO}
Let $(M,(\phi_s)_{s\in\R})$ be a transitive $C^2$ Anosov flow.  If there is no volume a.c.i.p.,  then $(M,\vol_M,(\phi_s)_{s\in\R})$ is totally dissipative. 
\end{theorem}
\begin{proof}

There exists a volume a.c.i.p.\ measure if and only if $\mu^+=\mu^-$ \cite{MR0317355}, so we have $\mu^+\neq \mu^-$.  Let $T=\phi_1$ be the time-one map of the flow.  Recall that it suffices to show that $T$ is totally  dissipative \cite[Corollary 1.6.5]{AaroBook}.    Moreover, note that  $T$ is totally dissipative if and only if $T^{-1}$ is totally dissipative.

Fix a  continuous function on $g:M\to \R$ with $\int gd\mu^+\neq \int gd\mu^-$. 
Set $\delta:=\tfrac{1}{2}{\left|\int gd\mu^+-\int gd\mu^-\right|}>0$.  Let $\epsilon>0$. As $\mu^+$ is a SRB measure, there exists $N>0$ such that the set
\[
A=\left\{x\in M:\ \forall n>N, \text{ we have } 
\left|\frac{1}{n}\int_0^n g\circ \phi_t(x)dt
-\int gd\mu^+\right|<\delta \right\},
\]
satisfies $\mu(A)>1-\epsilon$. For all $n\in\N$,  we have
\[
\left(\int_0^n g\circ \phi_tdt\right)\circ T^{-n}(x)=\int_{-n}^0 g\circ \phi_t(x)dt.
\]
As the latter tends  to $\int gd\mu^-$ as $n\to\infty$ for almost every $x \in M$,  for $\vol_M$-almost every $x\in A$, we have
\[
\#\{n\in\N:\ T^{-n}x\in A\}<\infty. 
\]
Since $\epsilon>0$ is arbitrary, it follows from Halmos recurrence theorem \cite[Theorem 1.1.1]{AaroBook} that $(M,\vol_M,T^{-1})$ is totally dissipative. 
\end{proof}

In order to use Theorem \ref{Dissipative flows} (a),    we will need the following lemma.

\begin{lemma} 
\label{space-atoms}
A transitive and totally  dissipative  $C^2$  flow $(\phi_s)_{s \in \R}$ on a compact manifold $M$ is isomorphic to the canonical translation flow of $\R$.
\end{lemma}

\begin{proof}
We already know from Theorem \ref{Krengel} that the  flow is isomorphic to the translation flow of $\R$, with respect to some standard probability space $(\Gamma, \mathcal{C}, \nu)$; thus it suffices to show that $\Gamma$ is non-atomic.   Towards a contradiction, let $\pi:\Gamma\times \R\to M$ be the isomorphism and $g \in \Gamma$ be a $\nu$-atom.  Since $\pi$ is an isomorphism,  there exists $x \in M$ such  that following its orbit on a manifold for one unit of time is the same as moving along the unit interval, so that 

\begin{align*}
(\nu\otimes \Leb)\circ\pi^{-1}(\ns{\phi_s(x): s \in [0,1]}) 
&=(\nu\otimes \Leb)(\{g\}\times [0,1])\\
&= \nu(g) \Leb([0,1]) >0.
\end{align*}
As $(\nu\otimes\Leb)\circ\pi^{-1}$ and $\vol_M$ are equivalent measures we have that $$\vol_M\ns{\phi_s(m): s \in [0,1]}>0.$$ We recall that the orbit  $s \mapsto \phi_s(x)$ is smooth and is the solution to first order autonomous ordinary differential equation \cite[page 795]{MR0228014}.  However,  it is well-known that the image of a smooth curve on a manifold, with dimension two or higher, has no volume \cite{MR7523}.  
\end{proof}

\begin{proof}[Proof of Theorem \ref{Anosov flow chaos}]
If the flow has a volume a.c.i.p.\ $\mu$, then  $\mu$ is a Gibbs measure \cite{MR0317355}.    It follows from a result of Ornstein and Weiss \cite{MR325926} and Ratner \cite{MR374387} that flow endowed with $\mu \sim \vol_M$ is isomorphic to a  Bernoulli flow. 

If there is no volume a.c.i.p.,  then by Theorem \ref{GO} the flow endowed with volume measure  is totally dissipative.   
 By Lemma \ref{space-atoms}, the flow is isomorphic to the canonical translation of $\R$, and hence by Theorem \ref{Dissipative flows} (a),    \emph{every} Bernoulli flow is a factor.
\end{proof}

\section{Near optimal Sinai factors}
\label{section-near}

\subsection{Markers and fillers}
\label{markers}
Let $\kmark \in \Z^{+}$ be a large positive integer which we will specify later.  Let $A$ be a finite set containing the two distinct  symbols $a$ and $b$.       Given $x \in A^{\Z}$, we call an integer interval $[n, n+2\kmark]$ a \dff{marker}, if $x_{n+i} =a$ for all $0 \leq i \leq 2\kmark -1$ and $x_{n+ 2\kmark} =b$.   Any integer that does not belong to a marker, belongs to a maximal  \dff{filler}, so that markers and fillers partition the integers.   Thus with markers and fillers, our task is to find an encoding of  a finite non-stationary sequence into a finite stationary sequence that retain most of the entropy.      A technical problem arises that we cannot control the size of the fillers, and for our construction, it will be convenient to have a version of  fillers of a fixed size.     The following idea is  from Ball \cite{ball} and its presentation is adapted from \cite{Qsc}.  We define a bi-infinite sequence of \dff{alternating} intervals $I(x) = (I_i)_{i \in \Z}$ that partition $\Z$ into
intervals of length $\kmark$ and $1$ in the following way.   Locate all the markers of $x$.
Any $n \in \Z$ that belongs to the right endpoint of a marker is an interval of length $1$,
following a marker will always be an interval of length $\kmark$, and if $x$ restricted to
the interval of length $\kmark$ is not a string of $\kmark$ consecutive $a$'s, then the following  interval will also be one of length $\kmark$, otherwise, the following intervals will
all be of length $1$, until the $a$ stops occurring; the following interval will be
one of length $\kmark$.  We say that a \dff{switch} occurs  in an alternating interval if it is an interval of length $\kmark$ and is a string of consecutive $a$'s or if its an interval of length $1$ and the symbol that is \emph{not} $a$ has appeared.    For definiteness, we require that $0 \in I_0$, and $\sup I_i < \inf I_j$ if $i < j$.    In what follows it will be more convenient to use the language of random variables.  

\begin{proposition}[Ball's alternating intervals]
\label{ball}
  Let $X = (X_i)_{i \in \Z}$ be random variables corresponding to the law of a Bernoulli shift on $A$.  Let $\kmark$ be a positive integer.   Conditioned on the alternating intervals $I(X) = (I_i)_{i \in \Z}$, the random sequence $X$ has the following properties:

\begin{itemize}
\item
The random variables $(X \ronn_{I_i})_{i \in \Z}$ are independent.  
\item
On each alternating interval $I_i$ of size $1$ not immediately left of an interval of size $\kmark$, we know that $X \ronn_{I_i} =a$; otherwise $X \ronn_{I_i} \not =a$, and a switch occurs.
\item
On each alternating interval $I_i$ of size $\kmark$ that is not immediately left of an interval of size $1$, the law of $X \ronn_{I_i}$, is the law of $X\ronn_ {I_i}$ conditioned not to be a string of $a$'s; otherwise $X \ronn_{I_i}$ is a string of $a$'s, and a switch occurs. 
\end{itemize}

\end{proposition}

\begin{proof}
Let $\rho$ be the law of $X$. Let $n \in \Z^{+}$.  Consider the random variables $Y^n = (Y_i)_{i =-n} ^{\infty}$ sampled using the procedure.  We start by sampling $\kmark$ elements from the measure $(\rho_{-n}, \ldots, \rho_{-n+k-1})$, and we continue to sample in blocks of $\kmark$ until a switch  occurs, in which case we sample one coordinate at a time, until a switch occurs, and then we go back to sampling $\kmark$ elements at a time.   Since switches are stopping times, it follows that $Y^n$ has the same law as $X^n=(X_{i})_{i=-n}^{\infty}$, for all $n \in \Z^{+}$.   The above properties clearly hold for the weak limit of $Y^n$ and thus hold for $X$. 
\end{proof}
\subsection{A Kakutani equivalent coding}
\label{earlier-dis}
Recall that by Kakutani's theorem \cite{MR23331}, we have that two infinite direct product measures $\mu$ and $\nu$ on $A^{\Z}$ satisfying the Doeblin condition are equivalent if and only if 
\begin{equation}
\label{KE-thm}
\sum_{n \in \Z} \sum_{a \in A}  (\mu_n(a)- \nu_n(a))^2<\infty.
\end{equation}

Note that Kakutani's theorem can be used to identify which Bernoulli shifts are nonsingular and which Bernoulli shifts are equivalent to stationary ones.     One can imagine a coin flipper who progressively gets better at flipping a coin, but does not get better so quickly that their flips are  indistinguishable from iid ones.

We will say that a statement holds for a product measure $\mu$ modulo or up to Kakutani equivalence if there is an equivalent measure $\nu$ for which the corresponding statement holds. 
A key idea in our proof of Theorem \ref{studia} is that although a Bernoulli shift may not be  given by independent and identical observations, non-identical observations can be combined in a way to yield identical observations, up to Kakutani equivalence.   Specifically, we adapted  von Neumann's observation  that one can simulate a fair coin  from a possibly biased coin, using what is now commonly referred to as   rejection sampling \cite{vonNeumann1951};  flip the biased coin twice: if we get 
HT, 
we report this as  
H and if we get 
TH, we report this as a 
T, and we are to repeat this procedure if we get either 
HH or TT.  We remark that in the nonsingular setting, we are faced with the added difficulty of using different coins for each flip.

Given a product measure $\rho=\bigotimes_{i\in\Z}\rho_i$, we let $\rho_i^{\inseq k}$ be the probability measure on $A^k$ distributed as $(\rho_i,\ldots,\rho_{i+k-1})$; we will use the notation $\rho_i^{\inprod k} = (\rho_i, \ldots, \rho_i)$ to denote the $k$-fold product of the measure $\rho_i$.   An important fact we will exploit is that the sequence of  measures $\rho_i^{\inseq k}$ and $\rho_i^{\inprod k}$ are equivalent,  in the sense of Kakutani; see Lemma \ref{lem: Kakuuu} and also \cite[Theorem 20]{KoSooFactors}.    Some of the proofs and lemmas are concerned with demonstrating that we can substitute a more difficult statement involving $\rho_i^{\inseq k}$ with a simpler  statement involving $\rho_i^{\inprod k}$.      If $Q$ is a probability measure  and $B$ is a measurable subset, then the probability measure given by 
$$ Q(\cdot |B) := \frac{Q(\cdot \cap B)}{Q(B)}$$
will sometimes be referred to as the \dff{conditional probability of $Q$ given $B$}.
%
%
\begin{proposition}[Kakutani equivalent coding]
	\label{KE-extract}
	Suppose $\rho$ is a nonsingular Bernoulli product measure that satisfies the Doeblin condition and has a limiting marginal measure ${p}$ on $A$.    Let $\e >0$.  There exists $\alpha>0$ and $\kmark$ sufficiently large such that for all $k\geq\kmark$ there exists a subset $\mathcal{G}_k \subset A^{k}$ omitting the sequence $a^{k} \in A^k$, and a finite set $B_k$ with the following properties:
	\begin{enumerate}[(a)]
	\item $\# B_k\geq 2^{k(H(p)-\e)}$.
	\item For all but finitely many $i\in\Z$, we have $\rho_i^{\inseq k}(A^k\setminus\mathcal{G}_k)\leq e^{-\alpha k}$. 
	\item   There exists a single  function $\psi : \mathcal{G}_k \to B_k$ such that
	\begin{equation}
	\label{KE-sum}
	\sum_{i \in \Z} \sum_{{c} \in B_k} \left(  \rho^{\inseq k}_i \big( \psi^{-1} ({c})|\mathcal{G}_k \big) - \frac{1}{\# B_k} \right)^2 < \infty. 
	\end{equation}
	\end{enumerate}
\end{proposition}

In our proof of Proposition \ref{KE-extract}, we make the following observation.   For simplicity consider the binary case where $A = \ns{0,1}$.   Each  $\rho_i^{\inseq k}$ can be  viewed as a disintegration given by first sampling from a distribution that gives the total number of ones  and then  sampling from a distribution that places the locations of the ones and zeros in the positions $(i, \ldots, i+k-1)$; in the case where $\rho$ is given by an identical product measure, the second distribution is given by a uniform distribution.   Since $\rho$ has a limiting distribution, by a large deviations argument, when $k$ is large, we can control the first distribution, so that we know up to an exponential error the number of ones that do occur,  and then it turns out even when $\rho$ is a nonsingular measure, the second distribution can be assumed to be uniform up to Kakutani equivalence.   We remark that in the iid case of  classical statistics, the first distribution corresponds to a  \emph{sufficient statistic} in the sense of Fisher \cite{fisher-suff}, which gives all the necessary information for estimating the parameter given by probability of an occurrence of a single one; in contrast,  we are more focused on the secondary uniform distribution, which contains no information about the parameter.

Before we give the details of the proof of Proposition \ref{KE-extract}, we show how it is used to prove Theorem \ref{near}.

\subsection{Using Proposition \ref{KE-extract}}

We say that an alternating interval of length $\kmark$ is \dff{good} if its values are in $\mathcal{G}_{\kmark}$.  Thus Proposition \ref{KE-extract}  allows us to replace $\kmark$ symbols that are asymptotically  distributed as $p^{\kmark}$ with a single random variable that is uniformly distributed on a set of size at least $2^{\kmark(H(p) - \epsilon)}$,  modulo Kakutani equivalence.   In our proof of Theorem \ref{near}, we will use Keane and Smorodinsky's finitary factor \cite{keanea} to generate a string of  $\kmark+1$  symbols  from the uniform random variables, so that at each good alternating interval of length $\kmark$ we will have $\kmark+1$ symbols, where the one \emph{extra} symbol can be possibly  distributed to integers that do not belong to a good alternating interval; the following lemma will be used to distribute the extra symbol.

\begin{lemma}[Matching]
\label{match}
Consider a  nonsingular Bernoulli shift that satisfies  the Doeblin condition and has a limiting measure.  Suppose each good alternating interval is assigned the colour green, an alternating interval of size $1$ assigned the colour red, and an alternating interval of size $\kmark$ that is not good assigned the colour maroon.     For $\kmark$ chosen sufficiently large, 
there is an equivariant (non-perfect)  matching of green to red and maroon intervals, where each red interval has $1$ green partner, and each maroon interval has $\kmark$ green partners.  
\end{lemma}

As in our previous constructions of nonsingular factors \cite{KosSoo20, KoSooFactors},  we will build upon a variant of the Me{\v{s}}alkin's  \cite{MR0110782} explicit  isomorphism of the stationary Bernoulli shifts $(\tfrac{1}{4},\tfrac{1}{4},\tfrac{1}{4},\tfrac{1}{4})$ and  $(\tfrac{1}{2},\tfrac{1}{8},\tfrac{1}{8},\tfrac{1}{8},\tfrac{1}{8})$, which is adapted from Holroyd and Peres \cite{MR2118858}. 

\begin{proof}[Proof of Lemma \ref{match}] The \dff{Me{\v{s}}alkin} matching has the following inductive description.  Let $W \in \ns{\text{red}, \text{maroon},\text{green}}^{\Z}$ be a random sequence of those colours; here think of $W$ as a colouring of the indexed alternating intervals $I$,  so that $W_i$ is the colour of $I_i$.     In the first instance,   if $W_n$ is red or maroon, and $W_{n+1}$ is green, then we match $n$ to $n+1$; that is, a red or maroon integer is matched to a green integer  that is to its immediate right.  Next, we remove all red and green integers that have $1$ partner, and maroon integers that have $\kmark$ partners, and repeat.  The  matching is \dff{successful} if all red integers  have a green integer partner, and all maroon integers have $\kmark$ green partners; the green partners will always be to the right of their red or maroon partners.     Note that by definition the resulting matching is equivariant with respect to the left-shift, and when applied to the coloured alternating intervals will also give an equivariant matching.   It suffices to show, via an elementary random walk argument, that we can choose $\kmark$ sufficiently large so that the excess of green integers compared to red and maroon integers will ensure that the Me{\v{s}}alkin matching is successful almost surely, when applied to the indexed alternating intervals.

Let $\rho$ be a nonsingular  Bernoulli shift on $A$ that satisfies the Doeblin condition and has a limiting measure.    Let $X$ have law $\rho$.   Without loss of generality assume that the symbols  $g,\hata,g' \not \in A$.       
Let $I(X)=(I_m)_{m \in \Z}$ be the alternating intervals.  If $I_m$ is of size $\kmark$, then let $J_m = g$ if it is good, $J_m=\hata$ if $X\ronn_{I_m} = a^\kmark$, and $J_m=g'$ if it is otherwise not good; 
if $I_m$ is of size $1$, let $J_m = X\ronn_{I_m}$.
By Proposition \ref{ball},   conditional on $I(X)$, the sequence $J_m$ is a non-homogeneous Markov chain on $A \cup \ns{g, \hata,g'}$ with transitions:
\begin{align*}
p_{gg}(t) &\geq  \min_{i \in \Z} \rho_i^{\inseq \kmark}( \mathcal{G}_\kmark ):=v_\kmark,\\
p_{g \hata}(t) + p_{g {g'}}(t)  &\leq 1- v_{\kmark}, \\
\sum_{c\in A}p_{\hata c}(t) &=1,\\
p_{g' g} & \geq v_{\kmark}, \\
p_{g' \hata} + p_{g' g'} & \leq 1 - v_{\kmark},	 \\
p_{aa}(t) & \leq \max_{i \in \Z} \rho_i(a) <1, \\
p_{ac}(t) & \geq 1-  \max_{i \in \Z} \rho_i(a) >0  \text{ for all } c \in A \setminus\ns{a},
\end{align*}
\begin{align*}
p_{cg}(t) & \geq v_{\kmark} \text{ for all } c \in A \setminus\ns{a}, \\
p_{cg'}(t) + p_{c \hata} & \leq 1- v_{\kmark}    \text{ for all }    c \in A \setminus\ns{a}.
\end{align*}
Thus only the state $g$ corresponds to a green interval, the other states all correspond to red or maroon intervals.  

For all $j<n$, let 
$$S_j^n:= \sum_{m=j}^ n \Big( \mathbf{1}[J_m=g]-\kmark\cdot\mathbf{1}{[J_m \in \{\hata, g'\}]} -\mathbf{1}[J_m \in A] \Big),$$
so that $S_j^n$ is an excess of the difference between green  intervals and red intervals with penalty $1$,  and maroon intervals, with penalty $\kmark$. 
%
Recall that by Proposition \ref{KE-extract} (b), the probability that an alternating interval of size  $\kmark$ is not good can be made exponentially  small and we   can replace $v_\kmark$, with a term $v'_{\kmark}$, such that   $\kmark\big(1-v'_{\kmark}\big)\to 0$ as $\kmark \to \infty$, at the cost  of the inequalities above  \emph{failing} for finitely many times $t \in \Z$.  A routine variation in a standard probabilistic argument in renewal theory  \cite{MR25098} gives that for $\kmark$ sufficiently large there exists $ 0 < C <1$ such that for any $j\ \in \Z$, we have 
\begin{equation}
\label{standard-LLN}
\P\Big( \liminf_{n \to \infty}\frac{1}{n-j+1}S_j^n \geq C \  | \  I(X) \Big) = 1.
\end{equation}
For each $j \in \Z$, if $I_j$ is red (or maroon) let $Z_j = n$ if $I_{n+j}$ is the (last $\kmark$) matched green interval under the  Me{\v{s}}alkin matching; if $I_j$ is green, set $Z_j=0$, and if $Z_j$ is red or maroon and  there are not enough green partners, then set $Z_j = \infty$.  
Let
$$ R_j = \inf\ns{ \ell\geq1: S_j^{j+\ell} >0}.$$
From the definition of the Me{\v{s}}alkin matching,
$$\P( Z_j >n \ | \ I(X) ) = \P( R_j >n \ | \ I(X) )$$
and the right hand side tends to zero by \eqref{standard-LLN}.  
Hence $Z_j$ is finite almost surely, and Me{\v{s}}alkin matching is successful almost surely. 
\end{proof}

\begin{proof}[Proof of Theorem \ref{near}]
Let $\varepsilon >0.$    Consider the following choice of parameter.
\begin{itemize}
\item
Let $\varepsilon' = \varepsilon/3$.     
\item
Choose $\kmark$ sufficiently large as to satisfy Lemma \ref{match} and Proposition \ref{KE-extract}, where in the notation of the proposition, $\e$ is replaced by $\e'$.
\item
Furthermore,  choose $\kmark$ sufficiently large  so that $$ H(p) \cdot \frac{\kmark}{\kmark+1} > H(p) - \varepsilon'.$$
\end{itemize}  

Let $X$ be a nonsingular Bernoulli shift with law $\rho$.    We define and identify markers, alternating intervals, and switches as in Section \ref{markers}.   Let $I(X)$ be the alternating intervals.   Let $\mathfrak{I} \subset I(X)$ be the good alternating intervals.    By Proposition \ref{ball}, it follows that conditioned on $\mathfrak{I}$, the random variables $(X \ronn_J)_{J \in \mathfrak{I}}$ are independent each with law $\rho_J(\cdot| \mathcal{G}_\kmark)$, given by  $\rho$ restricted to $J$ and conditioned to take values in $\mathcal{G}_\kmark$.

We apply Proposition \ref{KE-extract} to associate to each good alternating intervals  a single random variable  that up to  Kakutani equivalence, is uniformly distributed with entropy $$h >\kmark(H(p) - \e').$$  Furthermore, we apply Keane and Smorodinsky's  finitary factor \cite{keanea} to replace each uniform random variable by a  string of independent symbols from $B$, with law $q$,  of length $\kmark+1$, where $$(\kmark +1)H(q) > h - \e'.$$

Thus at each good alternating interval (of length $\kmark$) we have $\kmark +1$ symbols; the \emph{extra} symbol is distributed via the matching procedure given in Lemma \ref{match}, so that there is an independent symbol with law $q$  for every integer.    We disregard any remaining extra symbols that were not matched, and thus obtain an iid factor of entropy 
$$H(q) > \frac{h-\e'}{\kmark +1 }  >  (H(p) - \varepsilon') \cdot  \frac{\kmark}{\kmark +1 } - \e' > H(p) - \epsilon.$$    

We note that the external components involved in our construction: the  Me{\v{s}}alkin matching furnished by Lemma \ref{match} and the Keane and Smorodinsky factor are finitary.     Hence it follows, by definition, that our construction is also finitary.  
\end{proof}

\subsection{The proof of Proposition \ref{KE-extract}}

Generalizing our earlier discussion at the end of Section \ref{earlier-dis},  we note that the sufficient statistic for a finite number of iid observations of a categorical distribution  is given by the frequency counts of the types.  In what follows, we will use the theory of types to help prove Proposition \ref{KE-extract}.

\subsubsection{Types}

The \dff{empirical probability measure of $x\in A^k$}, denoted by $\emp(x)$, is the probability measure on $A$, given by
\[
\emp(x)(a)=\frac{1}{k}\sum_{j=1}^k \mathbf{1}[x_j =a] = \frac{\#\ns{j: x_j=a}}{k}.
\]
Given $k\in\Z^{+}$ we say that $p\in \PrA$ is of \dff{denominator $k$} 
if $kp(a)\in \N$ for all $a\in A$. The collection of probability distributions of denominator $k$ are precisely the ones which can arise as empirical probability measures for $x\in A^k$. 
For $p\in \PrA$ of denominator $k$, let $$\Ty_k(p)=\{x\in A^k: \emp(x)=p\}\subset A^k$$ be the \dff{$k$-type class} of $p$, which is all the sequences from $A$ of length $k$ which have an empirical measure that is equal to $p$.

We will use the following version of Proposition \ref{KE-extract} where we use conditioning to impose strict control over the types that can occur, so that a large deviations argument is not required.

\begin{proposition}\label{prop: Kakuu}
Suppose $\rho$ is  a nonsingular Bernoulli product measure that satisfies the Doeblin condition. Let $k\in\Z^{+}$.  Let  $p\in\PrA$ be of  denominator $k$ and $V\subset \Ty_k(p)$.  Recall that $\rho^{\inseq k}_{i}(\cdot|V)$ is the probability measure on $A^{k}$ given by taking $(\rho_i, \ldots, \rho_{i+k-1})$ conditioned to be on $V$.   Then for all $F\subset V$, we have
\begin{equation}
\label{KE-sum11}
\sum_{i \in \Z}  \left(  \rho^{\inseq k}_{i}(F|V) - \frac{\# F}{\#V} \right)^2 < \infty.
\end{equation}
\end{proposition}

Notice that if $\rho$ is an identical product measure, then Proposition \ref{prop: Kakuu} is not difficult since each summand is identically zero, see \eqref{Holroyd}.    The following lemma, which we will use to prove Proposition \ref{prop: Kakuu} connects this elementary observation to our nonsingular setting.


\begin{lemma}
\label{lem: Kakuuu}
Suppose $\rho$ is a nonsingular Bernoulli product measure that satisfies the Doeblin condition. For all ${c}\in A^k$, we have
\[
\sum_{i\in\Z}\left(\rho_i^{\inseq k}(c)-\rho_i^{\inprod k}(c)\right)^2<\infty.
\] 	
\end{lemma}
\begin{proof}
For all $c= (c_0, \ldots, c_{k-1}) \in A^k$, we have 
\begin{align*}
\rho^{\inseq k}_i(c)  &-  \rho_i^{\inprod k}(c) = \rho_i(c_0)\left(\prod_{j=1}^{k-1}\rho_{i+j}(c_j)-\prod_{j=1}^{k-1}\rho_{i}(c_j)\right)\\
&=\rho_i(c_0)\sum_{j=1}^{k-1} \left(\frac{\prod_{\ell=1}^j\rho_{i+\ell}(c_{\ell})}{\rho_{i+j}(c_j)}\right)\left(\rho_{i+j}(c_j)-\rho_i(c_j)\right)\left(\frac{\prod_{\ell=j}^{k-1}\rho_{i}(c_\ell)}{\rho_{i}(c_j)}\right).
\end{align*}
As $\rho$ satisfies the Doeblin condition, there exists $C>0$ such that for all $1\leq j\leq k-1$, we have
\begin{align*}
 \rho_i(c_0) \left(\frac{\prod_{\ell=1}^j\rho_{i+\ell}(c_\ell)}{\rho_{i+j}(c_j)}\right) &  \left|\rho_{i+j}(c_j) - \rho_i(c_j)\right|\left(\frac{\prod_{\ell=j}^{k-1}\rho_{i}(c_\ell)}{\rho_{i}(c_j)}\right)  \\
 &\leq C\left|\rho_{i+j}(c_j)-\rho_i(c_j)\right|.
\end{align*}
Consequently for all $i\in \Z$ and $c\in A^k$, we have
\[
\left(\rho_i^{\inseq k}(c)-\rho_i^{\inprod k}(c)\right)^2\leq C^2(k-1)^2 \sum_{j=1}^{k-1}\left(\rho_{i+j}(c_j)-\rho_i(c_j)\right)^2.
\]
As $\rho$ is nonsingular and satisfies the Doeblin condition, Kakutani's theorem implies that for all $m\geq 1$, we have
\[
A_m:=\sum_{i\in\Z}\sum_{j=1}^{m}\sum_{x\in A}\left(\rho_{i+j}(x)-\rho_i(x)\right)^2<\infty.
\]
Hence for all $c\in A^k$, we have
\begin{align*}
\sum_{i\in \Z}\left(\rho_i(c)-\rho_i^{\otimes k}(c)\right)^2&\leq C(k-1)^2\sum_{i\in\Z} \sum_{j=1}^{k-1}\left(\rho_{i+j}(c_j)-\rho_i(c_j)\right)^2\\
&\leq C(k-1)^2A_{k-1}<\infty. \qedhere
\end{align*}
\end{proof}
\begin{proof}[Proof of Proposition \ref{prop: Kakuu}]
Fix $V\subset \Ty(p)$ and $F\subset V$. A nice observation that is used in  \cite{MR2308235} is that for all $c,d\in \Ty_k(p)$, we have
\[
\rho_i^{\otimes k}(c)=\rho_i^{\otimes k}(d).
\]
Consequently for all $i\in \Z$, and ${c}\in V$, we have
\begin{equation}\label{Holroyd}
\rho_i^{\inprod k}({c} | V) = \frac{\rho_i^{\inprod k}({c})}{\rho_i^{\inprod k}(V)}=\frac{1}{\# V}.
\end{equation}
Since $V$ is a finite set it follows from Lemma \ref{lem: Kakuuu} that
\begin{align}
\label{lemma9-cons}
\sum_{i\in \Z}\left(\rho_i^{\inseq k}(V)-\rho_i^{\otimes k}(V)\right)^2&=\sum_{i\in \Z}\left(\sum_{{c}\in V} \rho_i^{\inseq k}({c})- \rho_i^{\inprod k}({c})\right)^2 \nonumber \\
&\leq (\#V)^2\sum_{{c}\in V}\sum_{i\in \Z}\left(\rho_i^{\inseq k}(c)-\rho_i^{\inprod k}(c)\right)^2<\infty. 
\end{align}

Fix ${c}\in V$. By the Doeblin condition for $\rho$, there exists $D>0$ such that for all $i\in\Z$, we have
\begin{align*}
\big[  \rho^{\inseq k}_{i}(c|V)  &- \rho_i^{\inprod k}({c} | V) \big]^2 = \\
& \left( \rho_{i}^{\inseq k} ({c})\left(\frac{1}{\rho_i^{\inseq}(V)}-\frac{1}{\rho_i^{\inprod k}(V)}\right)+ \frac{\rho_i^{\inseq k}({c})-\rho_i^{\inprod k}({c})}{\rho_i^{\inprod k}(V)}\right)^2\\
& \leq 4D \left[ \left(\rho_i^{\inseq k}(V)-\rho_i^{\inprod k}(V)\right)^2+\left(\rho_i^{\inseq k}({c})-\rho_i^{\inprod k}({c})\right)^2\right].
\end{align*}

Hence from \eqref{lemma9-cons} and  \eqref{Holroyd} we have,
\[
\sum_{i \in \Z} \left(  \rho^{\inseq k}_{i} ({c}|V) - \frac{1}{\# V} \right)^2<\infty,
\]
from which the desired result is immediate:
\begin{align*}
\sum_{i \in \Z} \left(  \rho^{\inseq k}_{i} (F|V) - \frac{\# F}{\# V} \right)^2&=\sum_{i \in \Z} \left( \sum_{{c}\in F}\left(  \rho^{\inseq k}_{i} (c|V) - \frac{1}{\#V}\right) \right)^2\\
& \leq (\#V)^2 \sum_{{c}\in F} \sum_{i \in \Z}  \left(  \rho^{\inseq k}_{i} (c|V) - \frac{1}{\#V} \right)^2<\infty.
\end{align*}
\end{proof}

\subsubsection{Sanov's theorem and the set $\mathcal{G}_k$ in Proposition \ref{KE-extract}}

In order to use Proposition \ref{prop: Kakuu}, we will need to introduce some results from large deviations theory.    Let $\delta >0$.  
%
%
%
Recall that the  \dff{total variation distance} between $q_1,q_2\in\PrA$ is defined by
\begin{equation}
\label{total-variational}
\dTV(q_1, q_2):=\sum_{a\in A}\left|q_1(a)-q_2(a)\right|;
\end{equation}
we will endow $\Prob(A)$ with this metric.   Recall that  $p=\lim_{|n|\to\infty}\rho_n$. For $k\in\mathbb{N}$, let 
\begin{equation}
\label{def:Udelta}
U(\delta):=\left\{q\in \PrA: \dTV(q,p)<\delta\right\}
\end{equation}
and
\[
\hat{\mathcal{G}}_{k,\delta}:=\left\{{x}\in A^k: \dTV(\emp(x), p)<\delta \right\}=\bigcup_{q\in U(\delta)} \Ty_k(q).
\]
Recall that the  \dff{Kullback-Leibler divergence} between $p,q\in\PrA$ is defined by 
\[
\DKL(q||p):=\sum_{a\in A}p(a)\log\left(\frac{p(a)}{q(a)}\right).
\]
The function $q\mapsto D(q||p)$ is a continuous function from $$\left\{q\in\PrA: q\ll p\right\} \text{ to } [0,\infty),$$
where $q\ll p$ means that $q$ is absolutely continuous with respect to $p$.  

The following is an adaptation of standard results on concentration of measure and well-known bounds of the method of types; see for example \cite[Chapter 11]{MR2239987}.
\begin{lemma}\label{lem: concentration}
For every $\varepsilon>0$ there exists $\delta>0$,  and a positive integer $k_0\in\Z^{+}$ and $\beta>0$ such that for all $k>k_0$:
\begin{enumerate}[(a)]
	\item 
	\label{a-ineq}
	If  $\dTV(q,p)<\delta$ and $q$ is of denominator $k$, then $\#\Ty_k(q)\geq 2^{\left(H(p)-\frac{\varepsilon}{2}\right)k}$.
		\item 
		\label{b-finite}
		For all but finitely many $n\in\Z$, we have 
		$\rho_n^{\inseq k}\left(\hat{\mathcal{G}}_{k,\delta}\right)\geq 1-e^{-\beta k}$.
\end{enumerate} 
\end{lemma}
\begin{proof}
By \cite[Theorem 11.1.3]{MR2239987}, for all $q\in \PrA$ and $k\in\Z^{+}$, if $q$ has denominator $k$, then
\[
\#\Ty_k(q)\geq \frac{1}{(k+1)^{|A|}}2^{kH(q)}.
\]
Since the entropy map  $q\mapsto H(q)$ is continuous, there exists $\delta>0$ such that if  $\dTV(q,p)<\delta$, then $H(q)> H(p)-\frac{\varepsilon}{3}$.  Let $k_1$ be such that if $k>k_1$, then $(k+1)^{|A|}<e^{\frac{\varepsilon}{6}k}.$
Hence, if $k>k_1$ and $\dTV(q,p)<\delta$, then $$\#\Ty_k(q)\geq 2^{\left(H(p)-\frac{\varepsilon}{2}\right)k},$$ establishing part \eqref{a-ineq}. 

The set $K:=\PrA\setminus U(\delta)$ is compact so that  
$$2\beta:=\min_{q\in K}\DKL(q||p)>0.$$ 
By Sanov's theorem \cite{sanov}, for every $k\in\Z^{+}$, we have
\[
p^{\inprod k}\left(A^k\setminus \hat{\mathcal{G}}_{k,\delta}\right)=p^{\inprod k}\left({x}\in A^k:\ \emp(x) \in K\right)\leq (k+1)^{|A|}e^{-2k\beta}. 
\]
Choose a positive integer $k_0\geq k_1$ such that for all $k\geq k_0$, we have
\begin{equation}\label{eq: Sanov}
p^{\otimes k}\left(A^k\setminus \hat{\mathcal{G}}_{k,\delta}\right)\leq \frac{e^{-\beta k}}{2}.
\end{equation}
Let $C:=\min_{a\in A}p(a)$. For every $n\in \Z$, and every  $k\geq k_0$ and $x\in A^k$, we have
\begin{align*}
\rho_n^{\inseq k}(x)=\prod_{j=n}^{n+k-1}\rho_j\left(x_j\right)&=\prod_{j=n}^{n+(k-1)}p\left(x_j\right)\left(1+\frac{\rho_j\left(x_j\right)-p\left(x_j\right)}{p\left(x_j\right)}\right)\\
&\leq p^{\inprod k}(x)\prod_{j=n}^{n+k-1}\left(1+\frac{\rho_j\left(x_j\right)-p\left(x_j\right)}{C}\right).
\end{align*}
Since $$\lim_{|n|\to\infty}\max_{x\in A^k}\prod_{j=n}^{n+k-1}\left(1+\frac{\rho_j\left(x_j\right)-p\left(x_j\right)}{C}\right)=1,$$  
for all but finitely many $n\in\Z$, we have
\begin{equation}
\label{proof-use}
\max_{x \in A}\left(\fracc{\rho_n^{\inseq k}(x)}{p^{\otimes k}(x)} \right)\leq 2.
\end{equation}
Hence with  \eqref{eq: Sanov},  we obtain that for all but finitely many $n\in\Z$, we have
\[
\rho_n^{\inseq k} \left(A^k\setminus \hat{\mathcal{G}}_{k,\delta}\right)\leq 2p^{\otimes k}\left(A^k\setminus \hat{\mathcal{G}}_{k,\delta}\right)\leq e^{-\beta k}.
\] 
as desired for part \eqref{b-finite}. 
\end{proof}

We now  combine  Proposition \ref{prop: Kakuu} with Lemma \ref{lem: concentration} to obtain Proposition \ref{KE-extract}.    Recall the role of  Proposition \ref{KE-extract} in  the proof of Theorem \ref{near} was to extract, using a single procedure,  (up to Kakutani equivalence) an iid sequence of sufficiently high entropy (discrete uniform) random variables from the good alternating intervals.   Lemma \ref{lem: concentration} gives that for $k$ sufficiently large most $k$-type classes will be large enough so that the random variable corresponding to picking an element from such a type class will have sufficiently high entropy.  Although there is a uniform lower bound on the sizes of the type classes, they vary.
Part of our proof of  Proposition \ref{KE-extract} will involve \emph{chopping} a type class up into equal sets of the desired size; exerting  further control over the sizes helps to prove that the resulting random variables are stationary.  See the proof below for details.  

\begin{proof}[Proof of Proposition \ref{KE-extract}]
Let $\e >0$.  Recall that $p$ is the limiting measure.     Choose $\delta$ and $k_0$ as in Lemma \ref{lem: concentration}  and  
$\kmark:=\max\big(k_0,\lceil\tfrac{2}{\varepsilon}\rceil\big).$ 
For all $k\geq \kmark$ set $B_k=\left\{1,\ldots,2^{\lceil k(H(p)-\varepsilon)\rceil  }\right\}$. We will first construct $\mathcal{G}_k$, and then define the mapping $\psi:\mathcal{G}_k\to B_k$. 
Set
$$U_k'(\delta):=   \ns{ q \in U(\delta):   q \text{ is of denominator $k$ and }  \Ty_k(q) \neq\emptyset},$$
where $U(\delta)$ is as given in \eqref{def:Udelta}.
 Let $q\in U'_k(\delta)\subset \PrA$. By Lemma \ref{lem: concentration}, since $k\geq \tfrac{2}{\epsilon}$, we have that $\#\Ty_k(q)>\#B_k$. Set $$m(q,k):=\left\lfloor \tfrac{\#\Ty_k(q)}{\#B_k}\right\rfloor.$$  
 For each $q\in U'_k(\delta)$, let  $F_k(q)$ be a fixed subset of $\Ty_k(q)$ of cardinality $m(q,k)(\#B_k)$; here we can think of \emph{chopping} up the set $\#\Ty_k(q)$ into $m(q,k)$ portions, stacked upon each other, and discarding away the rest.   Let
 \begin{equation}
 \label{GK}
 \mathcal{G}_k:=\biguplus_{q\in U'_k(\delta)}F_k(q)\subset \hat{\mathcal{G}}_{k,\delta}
 \end{equation}
 be given by a disjoint union.

In order to define $\psi:\mathcal{G}_k\to B_k$, note that since for all $q$ in the disjoint union \eqref{GK}, the cardinality of $F_k(q)$ is an integer multiple of the cardinality of $B_k$,  we  choose a $m(q,k)$-to-$1$ mapping  $\psi|_{F_k(q)}:F_k(q)\to B_k$ such that for all $c\in B_k$, we have 
\begin{equation}
\label{eq: psi in prop 6}
\frac{\#\left\{x\in F_k(q):\ \psi(x)=c \right\}}{\#F_k(q)}=\frac{1}{\#B_k};
\end{equation}	
putting together these choices, we obtain the desired map $\psi$.

We now verify the conditions (a),(b), and (c) of the proposition.  Condition  (a) holds by our choice of $B_k$. 

Now we verify that we did not chop too much away.  By construction, for all $q\in U'_k(\delta)$, by Lemma  \ref{lem: concentration} (a), we have
\begin{align}
\label{eq: F_k(q)}
\frac{p^{\inprod k}\left(\Ty_k(q)\setminus F_k(q)\right)}{p^{\inprod k}(\Ty_k(q))}&=\frac{\#(\Ty_k(q)\setminus F_k(q))}{\#\Ty_k(q)} \nonumber\\
&\leq \frac{\#B_k}{\#\Ty_k(q)}\ \leq\  e^{-\frac{\varepsilon k}{2}}.
\end{align}
In addition, from the proof of Lemma \ref{lem: concentration}, specifically inequality \eqref{proof-use},  for all but finitely many $n\in \Z$, for all $V\subset A^k$, we have 
$$\frac{1}{2}p^{\otimes k}(V)\leq \rho^{\inseq k}_n(V)\leq 2p^{\otimes k}(V).$$  
Consequently, for all but finitely many $n\in\Z$, we have
\begin{align*}
& \frac{\rho_n^{\inseq k}\left(\hat{\mathcal{G}}_{k,\delta}\setminus \mathcal{G}_k\right)}{\rho_n^{\inseq k}\left(\hat{\mathcal{G}}_{k,\delta}\right)} \leq 4 \cdot \frac{p^{\inprod k}\left(\hat{\mathcal{G}}_{k,\delta}\setminus \mathcal{G}_k\right)}{p^{\inprod k}\left(\hat{\mathcal{G}}_{k,\delta}\right)}  \\
&=\frac{4}{p^{\inprod k}\left(\hat{\mathcal{G}}_{k,\delta}\right)}\sum_{ {q\in U(\delta)} \atop {\Ty_k(q)\neq \emptyset} }\frac{p^{\inprod k}\left(\Ty_k(q)\setminus F_k(q)\right)}{p^{\inprod k}(\Ty_k(q))}p^{\inprod k}(\Ty_k(q)), \\
& \text{and by \eqref{eq: F_k(q)}, we obtain}  \\
&\leq \frac{4}{p^{\inprod k}\left(\hat{\mathcal{G}}_{k,\delta}\right)}\sum_{ {q\in U(\delta)} \atop {\Ty_k(q)\neq \emptyset} } e^{-\frac{\varepsilon k}{2}}p^{\inprod k}(\Ty_k(q)) \\
&\leq 4e^{-\frac{\varepsilon k}{2}}.
\end{align*}
Hence it follows from Lemma \ref{lem: concentration} (b) that  there exists $\beta >0$ such that for all but finitely many $n\in\Z$, we have
\[
\rho_n^{\inseq k} \left(A^k\setminus \mathcal{G}_k\right)=\rho_n^{\inseq k}\left(A^k\setminus \hat{\mathcal{G}}_{k,\delta}\right)+\rho_n^{\inseq k}\left(\hat{\mathcal{G}}_{k,\delta}\setminus \mathcal{G}_k\right)\leq e^{-\beta k}+4e^{-\frac{\varepsilon k}{2}}.
\]
By enlarging $\kmark$ if necessary, property (b) in Proposition \ref{KE-extract} holds for any $\alpha<\min(\beta,\frac{\varepsilon}{2})$. 

We will now prove that property (c) holds. 
%
Fix $c\in B_k$. Since $B_k$ is a finite set, it suffices to show that
\[
\sum_{i\in\Z} \left(  \rho^{\inseq k}_{i}( \psi^{-1} ({c})  | \mathcal{G}_k   ) - \frac{1}{\# B_k} \right)^2<\infty.
\]
 By Proposition \ref{prop: Kakuu} for all $q\in U'_k(\delta)$, we have with \eqref{eq: psi in prop 6} that
\begin{equation}\label{eq: kakkakkak}
\sum_{i\in\mathbb{Z}}\left[\rho^{\inseq k}_{i}\left(  \psi^{-1} ({c}) \cap  F_k(q) |  F_k(q)\right)  - 
\frac{1}{\#B_k}
\right]^2<\infty. 
\end{equation}
Recall by \eqref{GK}, by definition,  $$\sum_{q\in U'_k(\delta)}\frac{\rho_i^{\inseq k} \left(F_k(q)\right)}{\rho_i ^{\inseq k}(\mathcal{G}_k)}=1$$ 
and by \cite[Theorem 11.1.1]{MR2239987} or elementary counting, we have 
$$\#U'_k(\delta)\leq (k+1)^{|A|}.$$   
Thus for   all ${c}\in B_k$ and $i\in\mathbb{Z}$, we have 
\begin{align*}
&  \left(  \rho^{\inseq k}_i ( \psi^{-1} ({c})  | \mathcal{G}_k   ) - \frac{1}{\# B_k} \right)^2 \\
&= \left(\sum_{q\in U_k'(\delta)}  \frac{\rho_i^{\inseq k} \left(F_k(q)\right)}{\rho_i ^{\inseq k} (\mathcal{G}_k)}    \left[\rho^{\inseq k}_{i}\left(  \psi^{-1} (\bar{c}) \cap  
F_k(q) |  F_k(q)\right)  - \frac{1}{\#B_k}\right] \right)^2  \\
 &\leq (k+1)^{|A|}\sum_{q\in U'_k(\delta)} \left(\frac{\rho_i\left(F_k(q)\right)}{\rho_i(\mathcal{G}_k)}\right)^2 \left[\rho^{\inseq k}_{i}\left(  \psi^{-1} (\bar{c}) \cap  
F_k(q) |  F_k(q)\right)  - \frac{1}{\#B_k}\right] ^2\\
 &\leq (k+1)^{|A|}\sum_{q\in U'_k(\delta)} \left[\rho^{\inseq k}_{i}\left(  \psi^{-1} (\bar{c}) \cap  
F_k(q) |  F_k(q)\right)  - \frac{1}{\#B_k}\right] ^2.
\end{align*}
Hence summing over both sides of the inequality in the index $i\in \Z$, and applying  \eqref{eq: kakkakkak},  we obtain condition (c).
%
%
%
\end{proof}

\subsection{The proof of Corollary \ref{countable}}
\label{section-countable}

Let $A$ be a countable set.  Recall  that  a product measure $\rho=\bigotimes_{n\in\Z}\rho_i$ on $A^\Z$ has a limiting measure $p$ if $\rho_i$ converges to $p$ in the usual total variation distance; see \eqref{total-variational}.   

\begin{proof}[Proof of Corollary \ref{countable}]
Let $\rho$ be a nonsingular Bernoulli shift on a possibly countable set $A$ that has a limiting measure $p$.

We start with removing the Doeblin assumption for  the case of a  finite set $A$. 
\begin{equation}
\label{uniform-posss}
\text{Assume that for all $a\in A$, we have $p(a)>0$. }
\end{equation}
Since 
$\lim_{|i|\to\infty}\dTV(\rho_i, p)=0$,
there exists $N\in\Z^{+}$ such that for all $|i| \geq N$, we have
\[
\rho_i(a)>\frac{1}{2}\min_{a\in A}p(a)=\delta'.
\] 
Since $\rho$ is nonsingular, for all $i\in\Z$ and $a\in A$, we have $\rho_i(a)>0$. Hence the Doeblin condition is satisfied with
$$\delta:=\min\left(\min \ns{p_i(a):a\in A,|i|\leq N},\delta'\right),$$
and  Theorem \ref{near} applies.

Otherwise set $B=\{a\in A:\ p(a)>0 \}$. If $B$ is a singleton, then $H(p)=0$ and the Theorem is vacuously true.    Suppose that $B$ has more than one element.  Fix $b\in B$ and let  $\pi:A^\Z\to B^\Z$ be given by
\[
\pi(x)_j=\begin{cases}
b, &\ \text{if}\ x_j\in\{b\}\biguplus (A\setminus B),\\
x_j, &\ \text{if}\  x_j\in B\setminus\{b\}.
\end{cases}
\]
Since the value of $\pi(x)_j$ just depends on the coordinate $x_j$, it is easy to see that $\pi$ is a  finitary factor map from $\rho$ to $\nu:=\rho\circ \pi^{-1}$; furthermore $\nu$ is a product measure on $B^\Z$ with limiting marginal $p|_B$, and  $H(p|_B) = H(p)$.   Thus condition \eqref{uniform-posss} holds for $\nu$, and we already know it has near optimal entropy finitary iid factors, from which it follows by composition with $\pi$ that same holds for $\rho$.

Now with the Doeblin condition removed for the case of a finite set $A$, we consider the  countable case, where for concreteness we take $A=\N$.   The following cut-off functions will allow us to apply the result for the case of a finite set, established earlier.  For each $n \in \N$, consider  $\theta_n = \N \to \ns{0, \ldots, n}$ given by
$$ \theta_n(k)=\theta(k) = \min\ns{k,n},$$
and
$\Theta^n: \N^{\Z} \to  \{0,\ldots,n\}^\Z$ given by
\[
\Theta^n(x)_j=\Theta(x)_j=\theta(x_j) = \min\ns{x_j, n}.
\]
Clearly, $\Theta$ is a finitary factor map from $\rho$ to $\nu:=\rho\circ \Theta^{-1}$ and $\nu$ is a product measure on  $\{0,\ldots,n\}^\Z$ with  limiting measure $p\circ \theta_n^{-1}$; moreover, $H(p\circ \theta_n^{-1}) \to H(p)$, as $n \to \infty$.  Thus choosing $n$ finite and sufficiently large we again obtain   near optimal entropy finitary iid factors by composing $\Theta$ with the finitary factor that we obtained in the finite case.
\end{proof}

\section{Upper bounds on the entropy of a finitary factor}
\label{kalikow-section}

\subsection{Entropy rates for for bounding the  entropy of symbolic factors}

The next theorem is a mathematical abstraction of Theorem \ref{bound-Kalikow} that will allow for applications to Anosov diffeomorphisms and Bernoulli shifts.   

Let $A$ be a finite set.    For $x\in A^\Z$ and $n\in\Z^{+}$, we write for $M<n$ the set, 
$$[x]_M^n=\{y\in A^\Z:\ y_k = x_k \text{ for all } M \leq k \leq n\}$$ 
for the unique $[M,n]$-cylinder set  containing $x$.   Let  $\mu$ be a Borel regular measure on $A^\Z$ which is nonsingular with respect to the left-shift.  It will be convenient to use the language of random variables.   For example, recall that if $Z$ is a  discrete random variable or vector, then its Shannon entropy is given by
$$ H(Z) =     -\sum_{a} \P(Z=a)\log \P(Z=a).$$
Let  $(n_k)$ be a subsequence of  $\Z^{+}$ and $X \in A^{Z}$ be a random variable with law $\mu$, so that $\P(X \in \cdot) = \mu(\cdot)$.  
The \dff{lower entropy rate of the measure $\mu$ with respect to $(n_k)$} is given by
\[
h(\mu,(n_k)):=\liminf_{k\to\infty}\frac{H(X_1,\ldots,X_{n_k})}{n_k}.
\]
We recall that in the measure-preserving case, we may take $n_k = k$ and  subadditivity ensures the actual limit exists and is the Kolmogorov-Sinai entropy.  We will be able to use the lower entropy rate as a substitute for Kolmogorov-Sinai entropy in the nonsingular setting, provided the existence of certain physical proxies for $\mu$, which are akin 
to SRB measures.     

We say measure $\muphys$ on $A^\Z$ is \dff{a mean-physical measure for $\mu$ with respect to}  $(n_k)$, if for every cylinder set $C\subset A^\Z$, we have
\[
\lim_{k\to\infty}\frac{1}{n_k}\sum_{k=0}^{n_k-1}\mu\left(T^{-k}C\right)=\muphys(C).
\]
When $n_k=k$ we simply say that $\muphys$ is a \dff{mean-physical measure for $\mu$}.  We say  $\muphys$ is a \dff{physical measure for $\mu$} if for $\mu$-almost every $x\in A^\Z$, we have 
\[
\frac{1}{n}\sum_{k=0}^{n-1}\delta_{T^n x}\xrightarrow[n\to\infty]{}\muphys,\ \text{weakly}, 
\]
 where $\delta_y$ is the usual point mass giving unit mass to a set containing the point $y$ and zero mass otherwise. 
The portmanteau theorem  \cite[Theorem 2.4, page 87]{Durrett2} implies  that for every cylinder set $C\subset A^\Z$ and for $\mu$-almost every $x \in A^{\Z}$, we have
\[
\lim_{n\to\infty}\frac{1}{n}\sum_{k=0}^{n-1}1_C\circ T^k(x) = \muphys(C).
\]
It is easy to see that every physical measure is a mean-physical measure, and every mean-physical measure is shift-invariant.

\begin{theorem}
\label{Kalikow general}
	Let $\left(A^\Z,\B,\mu,T\right)$ be a nonsingular symbolic system.  If there exists $\muphys$, a mean-physical measure for $\mu$ with respect to a subsequence  $(n_k)$, then every iid system that is obtained as a finitary factor of $\mu$ has entropy no greater than $h(\mu, (n_k))$, the lower entropy rate of $\mu$ with respect to $(n_k)$.
\end{theorem}

 The following lemma will allow us to transfer the finitary assumption into a form involving finite union of cylinder sets, which will be useful in the proof of Theorem \ref{Kalikow general}.

\begin{lemma}
	\label{prop: al iz well}
	Let $\left(A^\Z,\B,\mu,T\right)$ be a nonsingular symbolic dynamical system, $\pi:A^\Z\to B^\Z$ a finitary nonsingular factor map from $\mu$ to $q^\Z$ and $\{B_s:\ s\in B\}$ be the partition of $B^\Z$ according to the zeroth coordinate. For every continuous probability measure $\nu$ on $A^\Z$ and $\epsilon>0$, there exists $\{C_s\}_{s\in B}$ and $C_\emptyset$ such that: 
	\begin{enumerate}[(a)]
		\item 
		For each $s\in B$,  the set $C_s\subset \B$ is a finite union of cylinder sets of $A^\Z$; the set $C_{\emptyset}$ is a complement of a finite union of cylinder sets.
		\item
		The sets $\left\{C_s\right\}_{s\in B\cup \{\emptyset\}}$ form a partition of $A^\mathbb{Z}$.
		\item
		For all $s\in B$,  we have $\mu\left(C_s \setminus \pi^{-1}B_s\right)=0$. 
		\item
		$\mu\left(C_\emptyset\triangle\biguplus_{s\in B} \left(\pi^{-1}B_s\setminus C_s\right)\right)=0$ and $\nu(C_\emptyset)<\epsilon$.
	\end{enumerate}
\end{lemma}
\begin{proof}
	As a consequence of the finitary property of $\pi$, for every $s\in B$, there exists a sequence $\{D_n(s)\}_{n\in \Z^{+}}$ of pairwise disjoint cylinder sets such that 
	\[
	\mu\Big(\pi^{-1}B_s\triangle\biguplus_{n \in \Z^{+}} D_n(s)\Big)=0.
	\]
	Since $\left\{\pi^{-1}B_s\right\}_{s\in S}$ is a partition of $A^\Z$ modulo $\mu$, we have 
	$$\biguplus_{s\in B}\biguplus_{n \in \Z^{+}} D_n(s) = A^{\Z}\  \bmod \mu.$$ 
	As $\nu$ is a continuous probability measure, for $N$ sufficiently large, we have 
	\[
	\nu\Big(\biguplus_{s\in B}\biguplus_{n=N+1}^\infty D_n(s)\Big)<\epsilon. 
	\]
	Set $C_\emptyset:=\biguplus_{s\in B}\biguplus_{n=N+1}^\infty D_n(s)$.  For each $s\in B$, let  $C_s:=\biguplus_{n=1}^ND_n(s)$.   The  lemma is immediate.  
\end{proof}

We will also require some elementary inequalities from information theory.  Recall that  Fano inequality \cite{MR0134389} gives that if  
$Z$ and $Z'$ 
are finite-valued ($A$-valued) random variables, defined on the same probability space and $p_e=\P(Z\neq Z')$, then 
\begin{equation}
\label{Fano}
H(Z|Z')\leq H(p_e,1-p_e)+p_e(\log( \#A)-1).
\end{equation}

In our proof of Theorem \ref{Kalikow general}, we will use Fano's inequality to compare the entropies of two finite random strings, one of which is an approximation of the other.   
\begin{lemma}
\label{prop: fano ineq  inhom}
  Consider the nonsingular system  $\left(A^\Z,\B,\mu,T\right)$ and let $\muphys$ be a mean-physical measure for $\mu$ with respect to the subsequence $(n_k)$.   Let $\epsilon >0$. Consider the set-up and notation of Lemma \ref{prop: al iz well}, take $\nu = \muphys$ and obtain the set $C_{\emptyset}$ with $\muphys( C_{\emptyset}) < \e.$
 Define $\beta: A^\Z \to  (B\cup \{\emptyset\})^\Z$ via 
$$\beta(x)_n=s  \text{ if and only if }  T^nx\in C_s.$$ 
Let $X \in A^{\Z}$ be a random variable with law $\mu$.  Set
$$p_k := \P(\pi(X)_k \not = \beta(X)_k) = \mu \circ T^{-k}(C_{\emptyset}).$$
  Then for all $n\in \Z^{+}$, we have 
\begin{align*}
H(\pi(X)_1,\ldots,\pi(X)_n) & \leq H(\beta(X)_1,\ldots,\beta(X)_n) \  + \\
& \sum_{k=1}^n H(p_k, 1-p_k)+(\log(\#B+1)-1)\sum_{k=1}^{n}p_k.
\end{align*}
\end{lemma}
\begin{proof}
The proof follows from a routine application of the chain rule for entropy and Fano's inequality. 
\end{proof}

\begin{proof}[Proof of Theorem \ref{Kalikow general}]
We will continue to use  the notation of Lemmas \ref{prop: al iz well} and  \ref{prop: fano ineq  inhom}. 
Let $\delta>0$.  It is elementary that  we may choose   $\epsilon>0$ in  Lemma \ref{prop: fano ineq  inhom} so that if $\limsup_{n\to\infty}\left(\frac{1}{n}\sum_{k=1}^np_k\right)<\epsilon$, then for all sufficiently large $n$, we have
\begin{equation}\label{eq: bound in kalaus}
\frac{1}{n}\sum_{k=1}^nH(p_k,1-p_k)+\frac{(\log(\#B+1)-1)}{n}\sum_{k=1}^np_k<\delta. 
\end{equation}
 Since $C_\emptyset$ is a (disjoint) finite union of cylinder sets and $\muphys$ is a mean-physical measure for $\mu$ with respect to  $(n_k)$,  we have
\[
\lim_{k\to\infty}\frac{1}{n_k}\sum_{j=1}^{n_k}p_j=\muphys(C_\emptyset)<\epsilon.
\]
By Lemma \ref{prop: fano ineq  inhom} and \eqref{eq: bound in kalaus}, for all sufficiently large $k$, we have
\begin{equation}\label{eq: 1}
\frac{H(\pi(X)_1,\ldots,\pi(X)_{n_k})}{n_k}\leq \frac{H(\beta(X)_1,\ldots,\beta(X)_{n_k})}{n_k}+\delta.
\end{equation}
Since $\pi$ is a factor map, and $\mu\circ \pi^{-1}\sim q^\Z$, it follows from a variation  of the  Shannon-McMillan-Breiman theorem \cite[Theorem 8]{MR586779} that
\begin{equation}\label{eq: 2}
\lim_{n\to\infty}\frac{H(\pi(X)_1,\ldots,\pi(X)_n)}{n}=H(q).
\end{equation}

Let $M$ be sufficiently large so that each finite union of cylinder sets $C_s$ for $s \in B \cup \ns{\emptyset}$ is a finite union of cylinder sets defined on the coordinates $[-M, M]$. Consequently,  $(\beta(X)_1,\ldots\beta(X)_n)$ is a function of $X_{-M}, \ldots, X_{n+M}$ and 
\begin{align*}
\frac{H(\beta(X)_1,\ldots,\beta(X)_{n_k})}{n_k} &\leq \frac{2M\log(\#B+1)}{n_k}+\frac{H(X_{1},\ldots,X_{n_k})}{n_k},
\end{align*}
where we have used the chain rule for entropy together with the fact that $\beta_i(X)$ take at most $\#B+1$ values.   Hence
\[
\liminf_{k\to\infty}
\frac{H(\beta(X)_1,\ldots,\beta(X)_{n_k})}{n_k}\leq \liminf_{n\to\infty}\frac{H(X_{1},\ldots,X_{n_k})}{n_k}=h(\mu,(n_k));
\] 
together with \eqref{eq: 1} and \eqref{eq: 2}, for an arbitrary $\delta >0$, we have
\begin{equation*}
H(q)=\lim_{k\to\infty}\frac{H(\pi(X)_1,\ldots,\pi(X)_{n_k})}{n_k}\leq \liminf_{k\to\infty}\frac{H(X_1,\ldots,X_{n_k})}{n_k}+\delta. 
\qedhere
\end{equation*}
\end{proof}

\subsection{Bernoulli shifts}
We will prove the following more general version of Theorem \ref{bound-Kalikow} from which the advertised result is immediate.  

For a nonsingular Bernoulli measure $\rho$, write
\[
h_{+}= h:=\liminf_{n\to\infty}\frac{1}{n}\sum_{k=1}^nH(\rho_k)\ \text{and}\ \  h_{-}=\liminf_{n\to\infty}\frac{1}{n}\sum_{k=-n}^{-1}H(\rho_k).
\]

\begin{theorem}
	\label{bound-Kalikoww}
If $q^\Z$ is a finitary factor of a nonsingular Bernoulli shift $\rho$,  then
\[
H(q)\leq \min(h_{-}, h_{+}).
\]
\end{theorem}

\begin{proof}[Proof of Theorem \ref{bound-Kalikow}]
We have the additional assumption of a limiting measure $p$, which, with the continuity of $H$, implies $h_{+}=h_{-}=H(p)$, from which the result is immediate from Theorem \ref{bound-Kalikoww}.
\end{proof}

\begin{proof}[Proof of Theorem \ref{bound-Kalikoww}]
Let $(n_k)$ be a subsequence such that  
$$\lim_{k\to\infty}\frac{1}{n_k}\sum_{j=1}^{n_k}H(\rmu_j)=h_{+}.$$  Consider the  sequence of measures on $A^\Z$ given by $\zeta_k=\frac{1}{n_k}\sum_{j=1}^{n_k}\rmu\circ T^{j}$. By the Banach-Alaoglu theorem  \cite{MR1455} there exists a further subsequence $m_\ell=n_{k_\ell}$ and a probability measure $\nu$ on $A^\Z$ such that $\zeta_{m_{\ell}}$ converges to $\nu$ weakly.  Since  cylinder sets are clopen, it follows that $\nu$ is a mean-physical measure for $\rho$ with respect to $(m_\ell)$. By Theorem \ref{Kalikow general}, we have
\begin{equation*}
H(p) \leq h(\rho, (m_\ell)) = h_{+}.
\end{equation*}
Finally, we note that a mapping is equivariant with respect to the left-shift $T$ if and only if it is also equivariant with respect to the right-shift $T^{-1}$.   Interchanging $T$ with $T^{-1}$ in the argument  gives $H(q)\leq h_{-}$ and thus $H(q)\leq \min(h_{+},h_{-})$ as desired. 
\end{proof}

\subsection{Anosov diffeomorphisms}
\label{section:anosov-bound}
We  recall symbolic dynamics for Anosov diffeomorphisms will allow us to employ Theorem \ref{Kalikow general}.   Let $V$ be a finite set.  
%
%
%
Given  an \dff{adjacency} matrix $\Ad \in \{0,1\}^{V \times V}$, the \dff{subshift of finite type} corresponding to $\Ad$ is defined by
\[
\Sigma_\Ad:=\left\{s\in V^\Z:\ \forall k\in\Z,\ \Ad_{s_k,s_{k+1}}=1  \right\},
\] 
where it is endowed with the usual left-shift $T$.     For more information on symbolic dynamics and shifts of finite type see \cite{MR1369092, MR1507944}.

\begin{theorem}[Symbolic dynamics from Sinai \cite{MR0233038}]
\label{sinai-sym}
Let $f\in \Diff^2(M)$ be a transitive $C^2$ Anosov diffeomorphism on a compact manifold $M$ and $\epsilon>0$. Then there exist a finite set $V$, an adjacency matrix $\Ad$,   and  a covering   $\mathcal{R} = \{R_v\}_{v\in V}$ of $M$ by closed sets of diameter less than or equal to $\epsilon$ such that
\begin{enumerate}[(a)]
	\item For  distinct $v,v'\in V$, the interiors of $R_v$ and $R_{v'}$ have no intersection.    
	\item The coding map $\pi: \Sigma_\Ad \to M$  given by
	$$\pi(s)=\bigcap_{n\in\Z}{R}_{s_n}$$ is  continuous, onto and finite-to-one map which is equivariant so that  all $s\in \Sigma_\Ad$, we have $f\circ \pi(s)=\pi\circ T(s)$; that is, $\pi$ is a finite-to-one semi-conjugacy of the topological dynamical systems. 
	\item For every $y\notin\bigcup_{n\in\Z}\bigcup_{v\in V}f^n\partial R_v$,  the set $\pi^{-1}(y)$ is comprised of  a single point of $\Sigma_\Ad$. 
\end{enumerate}
\end{theorem}

In the context of Theorem \ref{sinai-sym}, the sets $\mathcal{R}$ are referred to as a  \emph{Markov partition} and the subshift a \emph{topological Markov chain}; see also \cite{MR0257315,MR277003} and \cite{MR175106}.   By Theorem \ref{sinai-sym} if an iid system is an almost-surely continuous (finitary) factor of an Anosov diffeomorphism endowed with the natural volume measure, then it will also be a finitary factor of a symbolic system, and making it possible to apply Theorem  \ref{Kalikow general}.    We will need some technical lemmas to deal with the boundary of the Markov partition.

 Recall that $f$ has in addition an SRB measure $\mu_f$ which is an ergodic $f$-invariant measure such for all continuous function $\varphi:M\to M$, we have for $\vol_M$-almost every $y \in M$ that 
\[
\frac{1}{n}\sum_{k=0}^{n-1}\varphi\circ f^k(y)\xrightarrow[n\to\infty]{}\int\varphi d\mu_f.
\] 
This asymptotic condition has a useful reformulation in terms of weak convergence of measures. For $y\in M$, define a sequence of measures $$\nu_n^y:=\frac{1}{n}\sum_{k=0}^{n-1}\delta_{f^ky}.$$
The SRB property gives that $\vol_M$-almost every $y\in M$, the measure $\nu_n^y$ converges weakly to $\mu_f$ as $n\to\infty$. A point $y\in M$ is a \dff{generic point} for $\mu_f$ if $\nu_n^y$ converges weakly to $\mu_f$.  We will use this formulation in the proof of the following lemma.
\begin{lemma}
\label{claim}
Let $f$ be a transitive, $C^2$ Anosov diffeomorphism on a compact manifold $M$.  Consider symbolic dynamics for $f$ as given in Theorem \ref{sinai-sym}.    Then $\vol_M\big(\bigcup_{n\in\Z}\bigcup_{v\in V}f^n\partial R_v\big)=0$.
\end{lemma}
We note that in  Lemma \ref{claim}, a volume a.c.i.p.\ may not exist, making the proof a bit harder.     Our proof will involve revisiting some technical lemmas and arguments regarding Markov partitions that can be found in \cite{MR2423393}.   We will show that a point on the boundary fails to be generic with respect to an SRB measure.    We thank Yuri Lima for his interest in the lemma and   pointing out an important error in an earlier version   of its proof.    

\begin{proof}[Proof of Lemma \ref{claim}]

It suffices to show that $\partial R= \bigcup_{v\in V}\partial R_v$ has zero volume, since the volume measure is nonsingular, and the  union in question is countable.  

We recall that by  \cite[Lemma 3.11] {MR2423393} we may express the closed boundary in question as the   union
$\partial R= \partial R ^{\mathrm{u}} \cup \partial R ^{\mathrm{s}}$,
where these sets are often referred to as the \dff{unstable} and \dff{stable} boundaries; furthermore from \cite[Proposition 3.15] {MR2423393}, we have  containments, 
$$ f( \partial R ^{\mathrm{s}}) \subseteq \partial R ^{\mathrm{s}}  \text{ and }  f^{-1}( \partial R ^{\mathrm{u}}) \subseteq \partial R ^{\mathrm{u}}.$$

The stable containment implies that each point $z$ in the  stable boundary has a forward orbit-closure $\orb^+(z):=\overline{\ns{f^n(z): n \in \N}}$ that will have an empty intersection with some closed neighbourhood of each point that is not in the boundary.  Fix a point $y_0  \not \in \partial R$.    By  Urysohn's separation lemma \cite[Proposition 2.1.18]{borel}, for each $z \in  \partial R ^{\mathrm{s}}$   there exists $\Upsilon_{z}:M \to [0,1]$ such that $\Upsilon=0$ on $\orb^+(z)$ and $\Upsilon=1$ on a closed neighbourhood of $y_0$.  Since the SRB measures are fully supported, we have that  $\int \Upsilon_z d\mu_f >0$.    Similarly, the unstable containment implies an analogous statements for the backward orbit-closure for a point on the unstable boundary.     
 
From \cite{MR0399421}, let   $\mu_f$ and $\mu_{f^{-1}}$ be SRB measures so that on a subset $M'$ of full volume,  we have that  for all $y \in M'$  the weak convergences:
$$\frac{1}{n}\sum_{k=0}^{n-1}\delta_{f^ky} \xrightarrow[n\to \infty]{} \mu_f \quad \text {and} \quad
\frac{1}{n}\sum_{k=0}^{n-1}\delta_{f^{-k}y} \xrightarrow[n\to \infty]{} \mu_{f^{-1}}.$$
However by definition of $\Upsilon$, for   $z \in \partial R^{\mathrm{s}}$, for all $n \in \Z^{+}$ we have 
$$\frac{1}{n}\sum_{k=0}^{n-1}\Upsilon_z \circ f^k(z) = 0 < \int \Upsilon_z d \mu_f,$$
and the forward weak convergence fails for $z$, and it is not a generic point for $\mu_f$.   Similarly, if $z$ belongs to the unstable boundary, the backwards weak convergence  fails.  Hence if $z$ is in the boundary, one of the weak convergences fails, so that $\partial R \subseteq M\setminus M'$, and must have zero volume.     
\end{proof}

\begin{remark}
\label{remark-iso}
Lemma \ref{claim} gives that the symbolic coding $\pi$ in    Theorem \ref{sinai-sym}  is a nonsingular continuous almost everywhere (finitary) isomorphism between the symbolic dynamics with the left-shift $(\Sigma_\Ad,\B,\eta,T)$ and the Anosov system  $(M,\B(M), \vol_M,f)$, where $\eta=\vol_M\circ \pi$.     Note that $\Sigma_\Ad \subseteq V^{\Z}$ has unit measure under $\eta$.    

Notice that an SRB measure $\mu_f$ corresponds to a mean-physical measure for $\eta$ in the symbolic space given by $\eta_{\mathrm{ph}} = \mu_f \circ \pi$.  Thus Theorem \ref{Kalikow general} immediately gives that the entropy of any finitary factor is bounded by $h(\eta, (k))$.  \erk
\end{remark}

\begin{lemma}
\label{calculations}
With the notation of Remark \ref{remark-iso} in force, if the maximum  diameter of the Markov partitions is sufficiently small, then
$$h(\eta, (k)) \leq \min(h_{\mu_f}(f), h_{\mu_{f^{-1}}}(f)).$$
\end{lemma}

\begin{proof}[Proof of Theorem \ref{anosov.bound}]
In Theorem \ref{sinai-sym}, one can choose the diameter of the Markov partition to be sufficiently small.  Thus the proof is immediate from Remark \ref{remark-iso} and Lemma \ref{calculations}.
\end{proof}

\subsubsection{The proof of Lemma \ref{calculations}} 
We will execute entropy calculations with the   notation of Remark \ref{remark-iso} and also  the Markov partition of Theorem \ref{sinai-sym}.  These entropy calculations are somewhat more difficult as they involve two  measures simultaneously, only one of which is preserved by the transformation.  In this subsection, we will typically use `$x$' to denote an element of the symbolic space, and `$y$' and `$z$' to denote elements of the manifold.

Let $X \in \Sigma_\Ad$ be a random variable with  law $\eta$.  Given $k\in\Z^{+}$, define  $I_k: \Sigma_\Ad \to (0,\infty)$ via
\[
I_k(x)=-\log\P(X_0=x_0|X_{-1}=x_{-1},\ldots,X_{-k}=x_{-k}).
\]
Also, we recall the standard  notation that for a probability space $(\Omega, \F, \zeta)$ and a partition $\alpha$ of $\Omega$, where $\alpha(\omega)$ is the part to which $\omega$ belongs,  we have 
$$H_\zeta(\alpha) = -\sum_{a \in \alpha} \zeta(a) \log\zeta(a) \text{ and }
I_{\zeta} [\alpha](\omega) = -\log\zeta( \alpha(\omega) ),$$
so that
$$ \int I_{\zeta}[\alpha] d\zeta = H_\zeta(\alpha).$$
More specifically, for $y\in M$, we have 
\begin{equation}
\label{info-ex}
I_{\vol}[{\mathcal{R}|\vee_{i=1}^k f^{-i}\mathcal{R}}](y)=-\log\left(\frac{\vol\left(\vee_{i=0}^k f^{-i}\mathcal{R}(y)\right)}{\vol\left(\vee_{i=1}^k f^{-i}\mathcal{R}(y)\right)}\right).
\end{equation}

\begin{lemma}\label{lem: SRB info}
 Let $k\in\Z^{+}$.   For $\eta$-almost every $x\in \Sigma_\Ad$, we have 
\[
\frac{1}{n}\sum_{j=0}^{n-1}I_k\circ T^j(x)\xrightarrow[n\to\infty]{}\int I_{\vol}[{\mathcal{R}|\vee_{i=1}^k f^{-i}\mathcal{R}}]d\mu_f. 
\]
\end{lemma}
\begin{proof}
By the definition of $\eta = \vol_M \circ\pi$ as the lift of the volume, for a subset of full $\eta$-measure, for all $x\in \Sigma'_\Ad$,  we have
\[
\Xi_n^x:=  \frac{1}{n}\sum_{j=0}^{n-1}\delta_{f^j\pi(x)}\xrightarrow[n\to\infty]{}\mu_f,\ \text{weakly}.
\]
Let $x\in\Sigma'_\Ad$.    If $\pi(x)=y\in \bigcap_{i=0}^{k}f^{-i}R_{x_i}$, then
\[
I_k(x)=I_{\vol}[{\mathcal{R}|\vee_{i=1}^kf^{-i}\mathcal{R}}](\pi(x))=-\log\left(\frac{\vol\left(\cap_{i=0}^{k}f^{-i}R_{x_i}\right)}{\vol\left(\cap_{i=1}^{k}f^{-i}R_{x_i}\right)}\right).
\] 
We have
\begin{equation}\label{info}
\frac{1}{n}\sum_{j=0}^{n-1}I_k\circ T^j(x)=\frac{1}{n}\sum_{j=0}^{n-1}I_{\vol}[\mathcal{R}|\vee_{i=1}^kf^{-i}\mathcal{R}]\circ f^j(\pi(x)).
\end{equation}
  By Lemma \ref{claim}, the function $I_{\vol}[\mathcal{R}|\vee_{i=1}^k  f^{-i}\mathcal{R}]$ is  Riemann integrable, since it is discontinuous only on the boundaries of the Markov partition.   By the portmanteau theorem, 
 \begin{align*}
 \frac{1}{n}\sum_{j=0}^{n-1}I_{\vol}[{\mathcal{R}|\vee_{i=1}^kf^{-i}\mathcal{R}}]\circ f^j(\pi(x))&=\int I_{\vol}[\mathcal{R}|\vee_{i=1}^kf^{-i}\mathcal{R}]d \Xi_n^x\\
 &\xrightarrow[n\to\infty]{}\int I_{\vol}[{\mathcal{R}|\vee_{i=1}^k f^{-i}\mathcal{R}}]d\mu_f,
 \end{align*}
 from which the desired result follows from \eqref{info}. 
\end{proof}

Notice that in Lemma \ref{lem: SRB info}, we did not obtain convergence to $H_{\mu_f}(\mathcal{R}|\vee_{i={1}}^{k}f^{-i}\mathcal{R})$, but instead ended up with an expression that contains both the volume measure and its SRB measure; in order to replace this expression with one  involving only  the SRB measure, we will also need to make use of the description of an SRB measure as a Gibbs measure of the geometric potential, and refer to results from Bowen and Ruelle \cite{MR380889}.   Let $\varphi^{(\mathrm{u})}:M\to [0,\infty)$ be defined by $\varphi^{(\mathrm{u})}(y)=-\log\lambda(y)$,  where $\lambda(y)$ is the Jacobian of the linear map, given by
$D_f: E_y^\mathrm{u}\to E_{f(y)}^\mathrm{u}.$

By \cite[Lemma 4.1]{MR380889}, the map $\varphi^{(\mathrm{u})}$ is H\"{o}lder continuous and by \cite[Proposition 4.4]{MR380889}, $\mu_f$ is the unique equilibrium measure for $\varphi^{(\mathrm{u})}$ and 
\begin{equation}
\label{formula-23}
h_{\mu_f}(f)=-\int \varphi^{(\mathrm{u})}d\mu_f. 
\end{equation}
We also recall the \emph{first volume lemma} \cite[Lemma 4.2]{MR380889}.    
Fix a metric $d$ on $M$.     For $\epsilon>0$,   $n\in\N$,  and $z\in M$, consider the Bowen ball given by
\[
B_z(\epsilon,n):=\Big\{y\in M: \max_{0\leq j\leq n}d(f^jz,f^jy)\leq \epsilon\Big\}.
\] 
\begin{lemma}[First volume lemma \cite{MR380889}] 
 \label{vol-lemma}
Fix a Riemannian metric $d$ on $M$ so that the volume measure  $\vol$ is derived from $d$.
For all small $\epsilon>0$, there exists $C=C_\epsilon>1$ such that for all $z\in M$ and $n\in\N$, we have 
\[
  \frac{1}{C}\exp\left(\sum_{j=0}^n\varphi^{(\mathrm{u})}\circ f^j(z)\right) \leq 
\vol(B_z(\epsilon,n))  \leq C \exp\left(\sum_{j=0}^n\varphi^{(\mathrm{u})}\circ f^j(z)\right).
\]
\end{lemma}

In what follows,   we will always assume that the maximum diameter of the atoms of the Markov partition, say $\epsilon$, is small enough so that the Lemma \ref{vol-lemma} holds.

\begin{lemma}\label{lem: T is BT}
Under the assumption that the maximal diameter of the Markov partition is sufficiently small, we have
\[
\liminf_{n\to\infty}\frac{1}{n}\sum_{k=1}^n \int I_{\vol}[{\mathcal{R}|\vee_{i=1}^k f^{-i}\mathcal{R}}]d\mu_f \leq h_{\mu_f}(f).
\]
\end{lemma}

\begin{proof}
For $k\in\N$, set
\[
a_k:=-\int \log\left(\frac{\vol\left(\vee_{i=0}^k f^{-i}\mathcal{R}(y)\right)}{\mu_f\left(\vee_{i=0}^k f^{-i}\mathcal{R}(y)\right)}\right) d\mu_f(y),
\]
and the version of $a_k$,  where we start the join of the partitions at $i=1$:  
\[
b_k:=-\int \log\left(\frac{\vol\left(\vee_{i=1}^k f^{-i}\mathcal{R}(y)\right)}{\mu_f\left(\vee_{i=1}^k f^{-i}\mathcal{R}(y)\right)}\right) d\mu_f(y).
\]
A key observation is that since $\mu_f$ is $f$-preserving, for all $ k\in\Z^{+}$, we have $b_k =  a_{k-1};$
applying this key relation to \eqref{info-ex}, we have 
\begin{align*}
\int I_{\vol}[{\mathcal{R}|\vee_{i=1}^k f^{-i}\mathcal{R}}]d\mu_f&=\int I_{\mu_f}[{\mathcal{R}|\vee_{i=1}^k f^{-i}\mathcal{R}}]d\mu_f +a_k-a_{k-1}\\
&= H_{\mu_f}\left(\mathcal{R}\left|\vee_{j=1}^kf^{-j}\mathcal{R}\right.\right)+a_k-a_{k-1}.
\end{align*}
Consequently for all $n\in \Z^{+}$, we have 
\[
\frac{1}{n}\sum_{k=1}^n \int I_{\vol}[{\mathcal{R}|\vee_{i=1}^k f^{-i}\mathcal{R}}]d\mu_f=\frac{a_n-a_1}{n}+\frac{1}{n}\sum_{k=1}^n H_{\mu_f}\left(\mathcal{R}\left|\vee_{j=1}^kf^{-j}\mathcal{R}\right.\right).
\]
Since $\mu_f$ is $f$-invariant and $\mathcal{R}$ is a generating partition, 
$$ \lim_{k \to \infty}H_{\mu_f}\left(\mathcal{R}\left|\vee_{j=1}^kf^{-j}\mathcal{R}\right.\right) =h_{\mu_f}(f).$$ Therefore the second term on the right hand side converges to $h_{\mu_f}(f)$ as $n\to\infty$. Hence it remains to  show that 
$$a:=\liminf_{n\to\infty}\frac{a_n}{n}\leq 0.$$ 

Let $\epsilon$ be the maximal diameter of the atoms in the Markov partition so that for all $y\in M$ and $n\in\N$, we have the inclusion
$$\vee_{i=0}^n f^{-i}\mathcal{R}(y)\subset B_y(\epsilon,n).$$ 
This inclusion together with Lemma \ref{vol-lemma} imply that
\begin{align*}
\frac{1}{n}\int \log\vol\left(\vee_{i=0}^n f^{-i}\mathcal{R}(y)\right)d\mu_f(y)&\leq \frac{1}{n}\int \log\vol\left(B_y(\epsilon,n)\right)d\mu_f(y)\\
&\leq \frac{\log C_\epsilon}{n}+\frac{1}{n}\int \left(\sum_{j=1}^n\varphi^{(\mathrm{u})}\circ f^j(y)\right)\mu_f(y)\\
&=\frac{\log C_\epsilon}{n}+\int \varphi^{(\mathrm{u})}\mu_f\\
&=\frac{\log C_\epsilon}{n}-h_{\mu_f}(f),
\end{align*} 
where the first equality uses that $\mu_f$ is $f$-preserving, and the last equality is from \eqref{formula-23}.   

Since $\mathcal{R}$ is a generating partition, we have
\begin{eqnarray*}
h_{\mu_f}(f) &=& 
\lim_{n \to \infty} \frac{1}{n}H_{\mu_f}\left(\vee_{i=0}^n f^{-i}\mathcal{R}\right) \\
&=& -\lim_{n\to\infty} \frac{1}{n}\int {\log\mu_f\left(\vee_{i=0}^n f^{-i}\mathcal{R}(y)\right)}d\mu_f(y).
\end{eqnarray*}
Hence 
\begin{align*}
a &= -\limsup_{n\to\infty}\frac{1}{n}\left(\int {\log\vol\left(\vee_{i=0}^n f^{-i}\mathcal{R}(y)\right)}d\mu_f + H_{\mu_f}\left(\vee_{i=0}^n f^{-i}\mathcal{R}\right) \right) \\
&\leq -\limsup_{n\to\infty}\left(\frac{\log C_\epsilon}{n}-h(\mu_f) +  \frac{1}{n}H_{\mu_f}\left(\vee_{i=0}^n f^{-i}\mathcal{R}\right) \right)  \\
&\leq 0. 
\qedhere
\end{align*} 
\end{proof}

\begin{proof}[Proof of Lemma  \ref{calculations}]
Let $X$ be a random variable with law $\eta$.  By the chain rule for entropy, for every $n\in\Z^{+}$, we have
\begin{align}
\label{we-showed}
\frac{1}{n}{H(X_1,\ldots,X_n)}&=\frac{1}{n}H(X_1)+ \frac{1}{n}{\sum_{j=2}^n H(X_j|X_{j-1},\ldots X_1)} \nonumber\\
&\leq \frac{1}{n}{\sum_{j=1}^kH(X_j)}+ \frac{1}{n}{\sum_{j=k+1}^n H(X_j|X_{j-1},\ldots X_{j-k})} \nonumber\\
&\leq \frac{1}{n}{k\log(\# V)}+ \frac{1}{n}{\sum_{j=k+1}^n H(X_j|X_{j-1},\ldots X_{j-k})}.
\end{align}
Here, the second inequality uses that entropy can only decrease under further conditioning and  for the  last one we recall that $X_j$ takes values on the set $V$ values. For every $k<j\leq n$, we have
\[
H(X_j|X_{j-1},\ldots X_{j-k})=\int I_k\circ T^jd\eta, 
\]
and
\[
 \frac{1}{n}{\sum_{j=k+1}^n H(X_j|X_{j-1},\ldots X_{j-k})}=\int\left( \frac{1}{n}\sum_{j=k+1}^nI_k\circ T^j(x)\right)d\eta(x).
\]
The  bounded convergence theorem and Lemma \ref{lem: SRB info} give that
\[
\lim_{n\to\infty} \frac{1}{n}{\sum_{j=k+1}^n H(X_j|X_{j-1},\ldots X_{j-k})}=\int I_{\vol}[{\mathcal{R}|\vee_{i=1}^k f^{-i}\mathcal{R}}]d\mu_f .
\]
Hence from \eqref{we-showed}, for every $k \in \Z^{+}$, we have
\[
\limsup_{n\to\infty}\frac{1}{n}{H(X_1,\ldots,X_n)}\leq \int I_{\vol}[{\mathcal{R}|\vee_{i=1}^k f^{-i}\mathcal{R}}]d\mu_f. 
\]
Moreover, summing over the index $k$ in the inequality above,    by Lemma \ref{lem: T is BT} we have 
\[
\limsup_{n\to\infty}\frac{1}{n}{H(X_1,\ldots,X_n)}\leq \liminf_{n\to\infty}\frac{1}{n}\sum_{k=1}^n \int I_{\vol}[{\mathcal{R}|\vee_{i=1}^k f^{-i}\mathcal{R}}]d\mu_f\leq h_{\mu_f}(f).
\]

Thus we have $h(\eta, (k)) \leq h_{\mu_f}(f)$; applying a similar argument involving $T^{-1}$ and $f^{-1}$, we also obtain that $h(\eta, (k)) \leq h_{\mu_{f^{-1}}}(f^{-1}) = h_{\mu_{f^{-1}}}(f)$. 
\end{proof}

\section{A dissipative Bernoulli shift with no finitary factors}
\label{example-section}

Recall that Kakutani's theorem \cite{MR23331} gives that two infinite direct product measures $\mu$ and $\nu$ on $A^{\Z}$ are either equivalent or mutually singular, and they are equivalent if and only if 
$$\sum_{n \in \Z} \sum_{a \in A}  (\sqrt{\mu_n(a)} -\sqrt{ \nu_n(a)})^2<\infty.$$

\begin{lemma}
\label{example no finitary}
Let $\mu=\bigotimes_{n\in\Z} \mu_i$ be the product measure on $\{0,1\}^\Z$ with marginals given by 
$$\mu_{i}(0)=\frac{10}{\sqrt{|i|+2}}=1-\mu_i(1).$$  The associated  Bernoulli  shift  is nonsingular and totally dissipative action of a non-atomic measure space. 
\end{lemma}

The proof will require some calculations; in particular, to show that the shift is dissipative we will apply a sufficient condition from  \cite[Lemma 2.2]{KosKMaharam}, which requires verifying that 
\begin{equation}
\label{suff-veri}
\sum_{n\in\Z}\int \sqrt{\frac{d\mu\circ T^n}{d\mu}}d\mu<\infty.
\end{equation}

\begin{proof}[Proof of Lemma \ref{example no finitary}]
Since
\[
\sum_{n\in\Z}\min(\mu_n(0),\mu_n(1))=\sum_{n\in\Z}\frac{1}{\sqrt{1+|n|}}=\infty,
\]
it follows that the measure $\mu$ is non-atomic; see  \cite{MR537216}. The shift is nonsingular because
\begin{align*}
&\sum_{n\in\Z}\left(\left(\sqrt{\mu_n(0)}-\sqrt{\mu_{n-1}(0)}\right)^2+\left(\sqrt{\mu_n(1)}-\sqrt{\mu_{n-1}(1)}\right)^2\right)\\
&\leq 20\sum_{n\in\Z}\left(\frac{\sqrt[4]{|n|+2}-\sqrt[4]{|n-1|+2}}{(\sqrt[4]{|n|+2})\cdot(\sqrt[4]{|n-1|+2})}\right)^2<\infty.
\end{align*}

It remains to show that the shift is dissipative. By Kakutani's theorem, for all $n\in\Z^{+}$, for $\mu$ almost every $x\in \{0,1\}^\Z$, we have
\[
\frac{d\mu\circ T^n}{d\mu}(x)=\prod_{k\in\Z}\frac{\mu_{k-n}(x_k)}{\mu_k(x_k)}.
\]
As in \cite{VaesWahl}, since for all $0<a,b<1$, we have
 $$\sqrt{ab}+\sqrt{(1-a)(1-b)}\leq 1-\frac{(b-a)^2}{2},$$
and as $\mu$ is a product measure, for all $n\in\Z^{+}$,  we have
\begin{align*}
\int \sqrt{\frac{d\mu\circ T^n}{d\mu}}d\mu&=\prod_{k\in\Z}\int \sqrt{\frac{\mu_{k-n}(s)}{\mu_k(s)}}d\mu_k(s) \\
&=\prod_{k\in\Z}\left(\sqrt{\mu_k(0)\mu_{k-n}(0)}+\sqrt{\mu_k(1)\mu_{k-n}(1)}\right)\\
&\leq \prod_{k\in\Z}\left(1-\frac{(\mu_k(0)-\mu_{k-n}(0))^2}{2}\right)\\
&\leq \exp(-\tfrac{1}{2}\sum_{k\in\Z}\left(\mu_k(0)-\mu_{k-n}(0)\right)^2).
\end{align*}
For all sufficiently large  $n\in\Z^{+}$ and $0\leq k\leq n$,  we have
\[
\frac{n}{(\sqrt{2+k+n})(\sqrt{2+k+n}+\sqrt{2+k})}\geq \frac{1}{5}. 
	\]
Consequently,
\begin{align*}
\sum_{k\in\Z}\big(\mu_k(0) -&  \mu_{k-n}(0)\big)^2  \geq \sum_{k=n+1}^{2n}\left(\mu_k(0)-\mu_{k-n}(0)\right)^2\\
&=\sum_{k=0}^{n}\left(\frac{10}{\sqrt{k+2}}-\frac{10}{\sqrt{2+k+n}}\right)^2\\
&=\sum_{k=1}^{n}\frac{1}{2+k}\left(\frac{10n}{(\sqrt{2+k+n})(\sqrt{2+k+n}+\sqrt{2+k})}\right)^2\\
&\geq \sum_{k=1}^n\frac{4}{2+k}= 4(1+o(1))\log(1+n).
\end{align*}
Hence we have for all $n$ sufficiently large,  
\[
\int \sqrt{\frac{d\mu\circ T^n}{d\mu}}d\mu\leq \exp\left(-\tfrac{3}{2}\log(n+1)\right)=\frac{1}{(n+1)^{\tfrac{3}{2}}},
\]
a summable $p$-series.
\end{proof}

\begin{proof}[Proof of Proposition \ref{example-no}]
Immediate from Theorems \ref{thm: dissipative factors} and  \ref{bound-Kalikow}, and Lemma \ref{example no finitary}.
 \end{proof}

\section*{Acknowledgements}

Terry Soo thanks Anthony Quas for teaching him Kalikow's argument that the 3-shift is not a factor of the 2-shift, over dinner, when he was his postdoctoral fellow in Victoria.   We also thank Yuri Lima,  Benjamin Weiss, Omri Sarig, and Jon Aaronson  for  many enlightening discussions, and Ronnie Kosloff for explaining the relation of Theorem \ref{thm: dissipative factors} to Prigogine's work.

\bibliographystyle{abbrv}
\bibliography{embedding}
\end{document}